\documentclass{amsart}
\usepackage{graphicx}
\usepackage{mathtools}
\usepackage{hyperref}
\usepackage[table]{xcolor}
\usepackage{graphicx}
\usepackage{multirow}
\usepackage{subfigure} 
\usepackage{amsthm}
\usepackage{amsmath,amssymb,amsbsy}
\usepackage{lipsum}
\usepackage[all,cmtip]{xy}

\numberwithin{figure}{section}

\newcommand{\sL}{\mathcal{L}}

\newcommand{\sP}{\mathcal{P}}
\newcommand{\sR}{\mathcal{R}}
\newcommand{\sU}{\mathcal{U}}

\newtheorem{defi}{Definition}[section]
\newtheorem{thm}{Theorem}[section]
\newtheorem{cor}{Corollary}[section]
\newtheorem{lemma}{Lemma}[section]
\newtheorem{eg}{Example}[section]
\newtheorem{remark}{Remark}[section]
\DeclarePairedDelimiter\floor{\lfloor}{\rfloor}
\begin{document}

\title{$\psi$ class intersections on Hassett Spaces for genus $0$ with all weights $\frac{1}{2}$}

\author{Nand Sharma}
\address{Department of Mathematics, Colorado State University, CO, USA}
\email{sharma@math.colostate.edu}

\begin{abstract}
Hassett spaces are moduli spaces of weighted stable pointed curves. In this work, we consider such spaces of curves of genus $0$ with weights all $\frac{1}{2}$. These spaces are interesting as they are isomorphic to $\overline{M}_{0,n}$ but have different universal families and different intersection theory. We develop a closed formula for  intersections of $\psi$-classes on such spaces. In our main result, we encode the formula for top intersections in a generating function obtained by applying a differential operator to the Witten-potential. 
\end{abstract}

\maketitle

\section{Introduction}
The moduli space of algebraic curves of genus $0$ with $n$ marked points, $ \overline{M}_{0, n} $ (with the Deligne-Mumford compactification \cite{Deligne}) has been an important topic of research in algebraic geometry. These spaces provide an algebro-geometric tool to study how pointed rational curves vary in families, and are of fundamental importance in areas like Gromov-Witten theory and topological quantum field theories \cite{invitation}.

In \cite{HASSETT2003316}, Hassett constructed a new class of modular compactifications $ \overline{M}_{0, \mathcal{A}} $ of the moduli space ${M}_{0, n} $ of smooth curves with n marked points parameterized by an input datum $\mathcal{A}$, consisting of a collection $\mathcal{A} = (a_1 , \ldots , a_n )$ of weights $a_i \in \Bbb Q \cap (0, 1]$ such that $ a_1 + \ldots + a_n > 2$. We call these spaces $ \overline{M}_{0, \mathcal{A}} $ the Hassett spaces of rational curves.

A lot of work is being done on Hassett spaces including developing their tautological intersection theory and weighted Gromov-Witten theory, e.g.\ in \cite{alexeev_guy_2008}, \cite{weightedGW} and \cite{losev2000}.In this work, we contribute to the tautological intersection theory of a special case of such spaces- Hassett spaces of rational curves with weights all $\frac{1}{2}$, denoted $ \overline{M}_{0,\left( \frac{1}{2}\right)^n }$. These spaces provide for interesting spaces for combinatorial results in its intersection theory because of the symmetry of weights and its connections with intersection theory for $\overline{M}_{0,n}$.These spaces are interesting also because they are fine moduli spaces, are isomorphic to $\overline{M}_{0,n}$, but have different universal families and different intersection theory. Exploring these differences and developing some results in its tautological intersection theory is the contribution of this work.
For this work, the following notations are used: a $\psi$ class on $\overline{M}_{0,n}$ is denoted as $\psi_i$; a $\psi$ class on $ \overline{M}_{0,\left( \frac{1}{2}\right)^n } $ is denoted as $\bar \psi_i$, and the pullback of a $\psi$ class under the reduction morphism from $\overline{M}_{0,n}$ to $ \overline{M}_{0,\left( \frac{1}{2}\right)^n } $ is denoted $\hat \psi_i$.
     
In our first result \ref{arbintThm}, we develop a closed formula (\ref{arbintThmEq}) for the monomials in $\hat\psi$ classes in terms of cycles on $\overline{M}_{0,n}$. This closed formula is derived using the relation (\ref{psihatpsiCorEq}) between the $\hat \psi$ classes and $\psi$ classes on $\overline{M}_{0,n}$, in which $\psi_i$ is corrected by all boundary divisors where the $i$-th mark is on a twig that gets contracted when pushed forward to $ \overline{M}_{0,\left( \frac{1}{2}\right)^n } $. The proof uses this relation to obtain the $\hat \psi$ monomials as monomials in $\psi$ classes and boundary divisors on $\overline{M}_{0,n}$. So, the summands in the resulting expansion correspond to modified $\psi$ monomials on certain boundary strata on $\overline{M}_{0,n}$ that are the intersections of these boundary divisors.The dual graphs of these strata are all `forked' graphs - graphs with a `central' node and some `forks', e.g.\ figure (\ref{fig:image13}).We then establish a bijection between summands in the expansion corresponding to these graphs that we call `$\mathcal{P}$'-graphs and the unordered partitions of $[n]$, such that cardinality of each subset in the partition is either $2$ or $1$. The resulting formula (\ref{arbintThmEq}) has the pullback of monomials in $\bar \psi$ classes on $ \overline{M}_{0,\left( \frac{1}{2}\right)^n } $ as a sum of the intersections of monomials in $\psi$ classes and boundary strata corresponding to $\mathcal{P}$-graphs on $\overline{M}_{0,n}$. Then we derive two corollaries (\ref{mainThm} and \ref{mainThm1}) of this result to calculate the top intersections. These give the top intersections of $\hat \psi$ classes as a sum of top intersections of $\psi$ classes on $\overline{M}_{0,n}, \overline{M}_{0,n-1}, \overline{M}_{0,n-2}, \ldots $ with some multiplicities. We point out here that our corollaries (\ref{mainThm} and \ref{mainThm1}) can also be deduced from theorem $7.9$ in \cite{alexeev_guy_2008}. For our work, we develop specific and explicit closed formulas for our special case of all weights $\frac{1}{2}$ and base our combinatorial analysis closely on the structure of dual graphs.

The main theorem \ref{GenPot} of this work encodes the closed formula (\ref{mainThm}) for top intersections in a generating function $ G(\mathbf{t}) $ obtained by applying a differential operator to the Witten-potential $F(\mathbf{t}) $ \cite{kock}. This operator $\hat{\mathcal{L}}$ takes the form of an exponential partial differential operator and provides a very nice compact way to describe these top intersections. 
For this, we define `$\mathcal{P}_{k}$-graphs' that are obtained by replacing the $i$-th mark with $k_i$ on a $\mathcal{P}$-graph, where $k_i$ is the exponent of $\hat{\psi_i}$ in the $\hat \psi$ monomial (\ref{pkDef}). As expected, there is a surjection between $\mathcal{P}$-graphs and $\mathcal{P}_{k}$-graphs. Corresponding to these `$\mathcal{P}_{k}$-graphs', we write a new version of our closed formula in terms of these graphs (\ref{pkClosedformula}). Then we show a bijection between the summands in this formula and the summands in the coefficient of the appropriate term in $\hat{\mathcal{L}}  \left( F(\mathbf{t}) \right)  $. The resulting coefficient, as a sum of all these summands, corresponds to the top intersections of $\hat\psi$ classes.

The paper is organized as follows. In section \ref{secbkGrnd}, we give the background required for this work which consists of a brief introduction to $\overline{M}_{0,n}$, $\psi$ classes and Hassett Spaces, with some relevant theorems and lemmas on these topics. In section \ref{secClosed}, we prove our first result which gives the closed formula for the intersections of $\hat \psi$ classes. In section \ref{numer}, we give the results for top intersections, and encode the formula for top intersections in the generating function that we obtain by applying a partial differential operator to the Witten-potential.

\section{Background}\label{secbkGrnd}
For background on this work, the author has mainly used \cite{invitation}, \cite{renzo1}, \cite{kock}, \cite{mod} and introductory sections of \cite{AYALA2006307}. Here we will recall some selected facts we explicitly use in this work.
\subsection{$\overline{M}_{0,n}$}
We denote by $\overline{M}_{0,n}$ the moduli space of stable, $n$ pointed rational curves, with at worst nodal singularities. The {\bf boundary} of $\overline{M}_{0,n}$ is defined to be the complement of $M_{0,n}$ in $\overline{M}_{0,n}$. It consists of all points parameterizing nodal stable curves.

Given a rational, stable $n$-pointed curve $(C,p_1, \ldots, p_n)$, its {\bf dual graph}  is defined to have:
\begin{itemize}
\item a vertex for each irreducible component of $C$;
\item an edge for each node of $C$, joining the appropriate vertices;
\item a labeled half edge for each marked point, emanating from the appropriate vertex.
\end{itemize} 

Figure (\ref{fig:image14}) below gives an example of the dual graphs of some strata in $\overline{M}_{0,5}$. 
\begin{figure}[!htb]
\centering
  \includegraphics[width=4in]{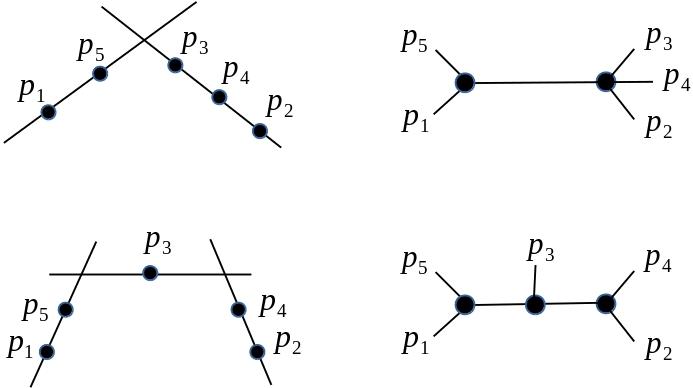}
\centering
        
  \label{fig:image14}
\caption{On the left are the boundary strata of $\overline{M}_{0,5}$, and the corresponding dual graphs are on the right. }
\end{figure}

The {\bf closures} of the codimension $1$ boundary strata of $\overline{M}_{0,n}$ are called the \textbf{irreducible boundary divisors};
they are in one-to-one correspondence with all ways of
partitioning $[n]=A\cup A^c$ with the cardinality of both $ A$ and  $A^c$ strictly greater than 1. We denote $D(A)= D(A^c)$ the divisor corresponding to the partition $\left\lbrace A,A^c\right\rbrace $. 
\subsection{$\psi$ classes}
For $i= 1, \ldots, n$, we define  the class $\psi_i \in A^1(\overline{M}_{0,n})$. 
\noindent Let $\textbf{L}_i \to \overline{M}_{0,n}$ be a line bundle whose fiber over each point $(C, p_1, \ldots, p_n)$ is canonically identified with $T^\ast_{p_i}(C)$. The line bundle $\bf{L}_i$ is called the $i$-th {\bf cotangent} (or {\bf tautological}) line bundle. Then
\begin{equation}
\psi_i := c_1(\bf{L}_i )
\end{equation} 
where $c_1$ is the first Chern class of the line bundle $\bf{L}_i$.

Some properties of $\psi $ classes on $\overline{M}_{0,n}$ that we use are the following lemmas. The interested reader can find their proofs in \cite{kock}. Here $[n]=I\cup I^c$ with the cardinality of both $ I$ and  $I^c$ strictly greater than $1$.
\begin{lemma}
Consider  the gluing morphism $
gl_I: \overline{M}_{0, I \cup \star} \times \overline{M}_{0, I^c \cup \bullet} \to  \overline{M}_{0, n} 
$. Assume that $i\in I$ and denote by $\pi_1:\overline{M}_{0, I \cup \star} \times \overline{M}_{0, I^c \cup \bullet} \to  \overline{M}_{0, I\cup \star}$ the first projection.
Then:
\begin{equation}
gl_I^\ast (\psi_i) =  \pi_1^\ast(\psi_i).
\end{equation}
\end{lemma}
\begin{lemma}\label{comparison}
Consider the forgetful morphism $\pi_{n+1}: \overline{M}_{0,n+1} \to \overline{M}_{0,n}$. Then, for every $i = 1, \ldots, n$,
\begin{equation}
\psi_i  = \pi_{n+1}^\ast(\psi_i) + D(\{i, n+1\}).
\end{equation}
\end{lemma}
\begin{lemma}
For any choice of $i, j, k$ distinct, we have the following equation in $A^1(\overline{M}_{0,n})$:
\begin{equation}
\psi_i = \sum_{i\in I, \  j,k \not\in I} D(I).
\end{equation}
\end{lemma}
\begin{lemma}[String Equation]\label{string}
Consider the forgetful morphism $\pi_{n+1}: \overline{M}_{0,n+1} \to \overline{M}_{0,n}$. Then
\begin{equation}
\pi_{n+1 \ast}\left(  \prod_{i=1}^n \psi_i^{k_i} \right)=
\sum_{j |k_j\not=0} \psi_j^{k_j-1}  \prod_{i\not=j} \psi_i^{k_i}.
\end{equation}
\end{lemma}
\begin{lemma}
Let $\sum k_i = n-3$. Then
\begin{equation}\label{stringEq}
\int_{\overline{M}_{0,n}} \prod_{i=1}^{n} \psi_i^{k_i} = {{n-3}\choose{k_1, \ldots, k_{n}}}
\end{equation}
where the integral sign denotes push-forward to the class of a point.
\end{lemma}
\subsection{Hassett spaces}
In \cite{HASSETT2003316}, Hassett constructed a new class of modular compactifications $ \overline{M}_{g, \mathcal{A}} $ of the moduli space ${M}_{g, n} $ of smooth curves with n marked points parameterized by an input datum $(g, \mathcal{A})$. Here $g$ is the genus of the curves and $\mathcal{A} = (a_1 , \ldots , a_n )$ is the {\bf weight data}  of weights $a_i \in \Bbb Q \cap (0, 1]$ satisfying the inequality $2g - 2 + a_1 + \ldots + a_n > 0$.

$ \overline{M}_{g, \mathcal{A}} $ that we call Hassett space parameterizes curves $(C, p_1 ,\ldots , p_n )$ with n marked non-singular points on C that are $\mathcal{A}${\bf-stable} if the the following two conditions are fulfilled.
\begin{enumerate}
\item The twisted canonical divisor $K_C + a_1 p_1 + \ldots + a_n p_n$ is ample.
\item A subset $p_{i_1} , \ldots, p_{i_k}$ of the marked points is allowed to coincide only if the inequality $a_{i_1} +\ldots+ a_{i_k} \leq 1$ holds.
\end{enumerate}
For $g=0$, the stability condition means that a rational $n$-pointed  curve $(C, p_1, \ldots, p_n)$ is $\mathcal{A}$-stable if on every irreducible component of $C$ the number of nodes plus the sum of the weights of the marks lying on that component is strictly greater than $2$; and on $C$, $a_1 + \ldots + a_n > 2$. 
In the case $(a_1 ,\ldots , a_n ) = (1, \ldots , 1)$, this condition is nothing but the traditional notion of an n-marked stable curve, and so the compactification $ \overline{M}_{g, \mathcal{A}} $ is exactly the well-known
Deligne-Mumford compactification $ \overline{M}_{g, n} $ of $ {M}_{g, n} $ in this case.
\begin{defi}
Given two weight data $\mathcal{A}, \mathcal{B}$, we say that $\mathcal{B} \leq \mathcal{A}$ if for every $i$, $b_i\leq a_i$. Then there exists a regular {\bf reduction} morphism:
\begin{equation}
c_{\mathcal{B},\mathcal{A}}:\overline{M}_{0, \mathcal{A}} \to \overline{M}_{0, \mathcal{B}}
\end{equation}
such that $c_{\mathcal{B},\mathcal{A}}(C,p_1, \ldots, p_n)$ is obtained by contracting twigs that become unstable when the weights of the points are ``lowered'' from $a_i$ to $b_i$.
\end{defi}
Moduli spaces of weighted stable rational curves also have $\psi$ classes, which are defined in the same way as for $\overline{M}_{0,n}$. A $\psi$ class on $ \overline{M}_{0, \mathcal{A}} $ will be denoted as $\bar \psi_i$.  
\begin{lemma}\label{redpsi}
Consider the reduction morphism
$c:  \overline{M}_{0,n} \to  \overline{M}_{0,\mathcal{A}}$. For $i = 1, \ldots, n$, we have:
\begin{equation}\label{redpsiEq}
\psi_i = c^\ast (\bar \psi_i) + \sum_{{i \in I,   a_{I} \leq 1}} D(I),
\end{equation}

where 
$
a_{I} = \sum_{j\in I}a_j.
$
\end{lemma}
\begin{proof}
Consider the following commutative diagram:
$$
\xymatrixrowsep{0.6in}
\xymatrixcolsep{1in}
\xymatrix{
        \overline{\sU}_{0,n} \ar[r]^C \ar[d]_\pi & \overline{\sU}_{0,\mathcal{A}}  \ar[d]_{\bar{\pi}}\\
         \overline{M}_{0,n} \ar[r]^c \ar@/_/[u]_{\sigma_i}  & \overline{M}_{0,{\mathcal{A}}} \ar@/_/[u]_{\bar{\sigma_i}}        }
$$
Here $ \overline{\sU}_{0,n} $ and $ \overline{\sU}_{0,\mathcal{A}} $ are the universal families over $\overline{M}_{0,n}$ and $\overline{M}_{0,\mathcal{A}}$ respectively; $ \sigma_i $ and $\bar \sigma_i $ are the $i$-th tautological sections of the corresponding universal families; and $c$ and $C$ are the reduction morphisms. 

Now, pulling back relative differentials, $\Omega_{{\pi}}$ and $\Omega_{\bar{\pi}} $,  yields a natural map
$$ C^{\ast} \Omega_{\bar{\pi}} \to \Omega_{{\pi}} . $$
So we get an induced map
$$ \sigma_i^{\ast} C^{\ast} \Omega_{\bar{\pi}} \to \sigma_i^{\ast} \Omega_{{\pi}}.$$
Since $C \circ \sigma_i = {\bar{\sigma}}_i \circ c$, we may re-write the map above as
$$ p : c^{\ast} {\bar{\sigma}}_i^{\ast} \Omega_{\bar{\pi}} \to \sigma_i^{\ast} \Omega_{{\pi}}.$$
Note that both sides are line bundles on $\overline{M}_{0,n}$, and their first Chern classes
are $c^{\ast} \bar{\psi_i}$ and $\psi_i$, respectively. An easy local calculation shows that the
map $p$ is an isomorphism on the interior and has a simple zero along the boundary divisor $D(I)$ if $i \in I$ and a $a_I \leq 1$. By comparing the first Chern
classes, we deduce
$$
\psi_i = c^\ast (\bar \psi_i) + \sum_{{i \in I,   a_{I} \leq 1}} D(I)
$$
as asserted.
\end{proof} 
Informally, for the pullback of a $\bar{\psi_i}$, a $\psi_i$ is corrected by all boundary divisors where the $i$-th mark is on a twig that gets contracted by $c$. 
\begin{defi}\label{psihatDefred}
Define the $ \hat{\psi_i}  $ class as the pullback of a $\bar\psi$ class under the reduction morphism $c:  \overline{M}_{0,n} \to  \overline{M}_{0,\mathcal{A}}$, 
$$
\hat{\psi_i} := c^{\ast} \bar{\psi_i}.
$$ 
\end{defi}
\begin{cor}\label{psihatpsiCor}
For the reduction morphism
$c:  \overline{M}_{0,n} \to  \overline{M}_{0,\mathcal{A}}$, where $ \mathcal{A} = \left\lbrace \frac{1}{2}, \frac{1}{2}, \ldots ,\frac{1}{2} \right\rbrace  $, for $i = 1, \ldots, n$, we have:
\begin{equation}\label{psihatpsiCorEq}
\hat \psi_i = \psi_i - \sum_{j, j \neq i} D\left(\left\lbrace  i,j\right\rbrace  \right) .
\end{equation}
\end{cor}
\begin{proof} 
This follows from (\ref{redpsiEq}) by observing that $ a_i + a_j = 1 $ for all $i,j$ when $ a_i = \frac{1}{2}$ $\forall$ $ i$. 
\end{proof} 
For the work that follows, $\mathcal{A}=(\frac{1}{2},\frac{1}{2}, \ldots, \frac{1}{2})$ for $\overline{M}_{0,\mathcal{A}}$, which we denote by $\overline{M}_{0,\left( \frac{1}{2} \right)^n  }$.
\section{Closed Formula for intersections of $\psi$-classes on $\overline{M}_{0,\left( \frac{1}{2} \right)^n}$}\label{secClosed}
In this section, we develop the closed formula for integrals of $\hat \psi$ monomials corresponding to the $\bar \psi$ monomials on $\overline{M}_{0,\left( \frac{1}{2} \right)^n  }$.
 
We denote by $\sP = \{P_1,P_2,..,P_j,S_1,S_2,...,S_i   \}$ an unordered partition of $[n]$, such that cardinality of each subset in the partition is either $2$ or $1$; if the cardinality of such a subset is $2$, we denote it with $P_j$ and if the cardinality of such subset is $1$, we denote it with $S_j$ . The elements of $P_j$ are denoted by $P_{j1}$ and $P_{j2}$. Further, we denote by $F$ the set of all $P_j$'s, and by $S$ the set of all $S_j$'s. The set of all such partitions is denoted by $\mathfrak{P}$.  
\begin{defi}
Given a $\sP \in \mathfrak{P}$, we define the graph $\Gamma_{\sP}$ as follows:
\begin{enumerate}
\item $\Gamma_{\sP}$ has a `central' node, with $|F|$ number of edges with nodes on ends
\item Attach to each non-`central' node two half edges forming a `fork'; to the `central' node, attach $|S|$ number of half-edges
\item Label the half-edges on a fork as $P_{j1}$ and $P_{j2}$, and half-edges on central node as $S_j$'s.
\end{enumerate}
So, each $P_j$ corresponds to a fork and $S_j$'s correspond to half-edges on the central node. We call this a \textit{$\sP$-graph}. $|F|$ gives the number of forks on the graph. Each such graph is a dual graph of a stratum $\Delta_\sP$ in $\overline{M}_{0,n}$. 
\end{defi}
Clearly, the set of all \textit{$\sP$-graphs} as defined above are in bijection with the set of all partitions $\sP \in \mathfrak{P}$.
\begin{defi}\label{coloringDef} 
Given a $\hat\psi$-monomial $m = \hat\psi_1^{k_1}\hat \psi_2^{k_2}\hat \psi_3^{k_3}...\hat \psi_r^{k_r}$, a \textit{decorated $\sP$-graph} $\Gamma_{\sP}^d$is obtained by coloring a half-edge corresponding to point $t \in P_j$ or $t \in S_j$ if $k_t \neq 0$. 
\end{defi}
\begin{eg}
Given $m = \hat \psi_1^{2}\hat \psi_2$ on $ \overline{M}_{0,6 } $, \\
for $\sP = \{1,2\},\{3\},\{4\},\{5\},\{6\}$, we get the decorated graph $\Gamma_{\sP}^d$ as in figure \ref{fig:image1};\\ for $\sP = \{1,3\},\{2,4\},\{5\},\{6\}$, we get the decorated graph $\Gamma_{\sP}^d$ as in figure \ref{fig:image2},\\ and for $\sP = \{1,3\},\{5,6\},\{2\},\{4\}$, we get the decorated graph $\Gamma_{\sP}^d$ as in figure \ref{fig:image3}.
\begin{figure}[!hbt]
\centering
\minipage{0.32\textwidth}
\centering
  \includegraphics[width=0.6in]{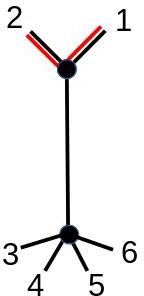}
        \caption{$\Gamma_{\sP}^d$}
  \label{fig:image1}
\endminipage\hfill
\centering
\minipage{0.32\textwidth}
  \includegraphics[width= 1.4in]{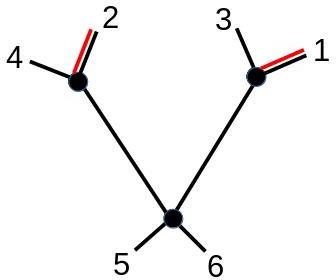}
        \caption{$\Gamma_{\sP}^d$}
  \label{fig:image2}
\endminipage\hfill
\minipage{0.32\textwidth}
\centering
  \includegraphics[width=1.5in]{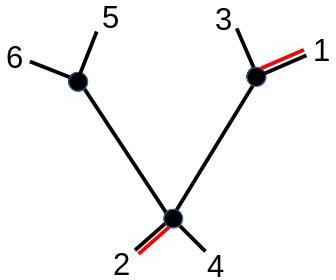}
        \caption{$\Gamma_{\sP}^d$}
  \label{fig:image3}
\endminipage\hfill
\end{figure}
\end{eg}
\begin{defi}\label{defpsibullet} 
Consider a stratum $\Gamma_\sP$, corresponding to the image of the following gluing morphism:
$$
gl: \overline{M}_{0,{P_1} \cup \star P_1 } \times  \ldots \times \overline{M}_{0,{P_s} \cup \star P_s } \times \overline{M}_{0,{S}\cup \bullet P_1 \cup \ldots  \cup \bullet P_s } \to  \overline{M}_{0, n} ,
$$
where for each $P_j$, the half-edge $\bullet P_j$ is the mark on the node pulled back from the factor corresponding to the central node, and $\star P_j$ is the mark on the node pulled back from the factor corresponding to the respective fork corresponding to a $P_j$. 

Consider the projection $f$ to the last factor in the product  
$$
f: \overline{M}_{0,{P_1} \cup \star P_1 } \times  \ldots \times  \overline{M}_{0,{P_s} \cup \star P_s } \times \overline{M}_{0,{S}\cup \bullet P_1\cup  \ldots  \cup \bullet P_s } \to  \overline{M}_{0,{S}\cup \bullet P_1\cup  \ldots  \cup \bullet P_s }. 
$$
We define:
\begin{equation}\label{psibu}
\prod_{j=1}^s\psi_{\bullet P_j}^{k_j}[\Gamma_\sP] :=gl_\ast f^\ast \left( \prod_{j=1}^s\psi_{\bullet P_j}^{k_j} \right)[\Gamma_\sP].
\end{equation}
\end{defi}

\begin{remark}\label{remcal}
The notation introduced  in  Definition \ref{defpsibullet} is an abuse which is convenient enough to be commonly accepted in the field. The class defined in \eqref{psibu} is not a product of the class of the stratum $[\Gamma_\sP]$ with another class in $\overline{M}_{0,n}$; it is a class supported on the cycle $\Gamma_\sP$, and at a basic level one may think of  the symbol $\prod_{j=1}^s\psi_{\bullet P_j}^{k_j}[\Gamma_\sP] $ just as a reminder of this fact. However, the notation becomes useful when performing intersection theoretic computations on such classes. Consider two such classes  $\prod_{j=1}^s\psi_{\bullet P_j}^{k_j}[\Gamma_\sP] $ and $\prod_{j=1}^{\tilde{s}}\psi_{\bullet \tilde{P}_j}^{\tilde{k}_j}[\Gamma_{\tilde{\sP}}]$. One may show that their product is  $\prod_{j=1}^s\psi_{\bullet P_j}^{k_j}\prod_{j=1}^{\tilde{s}}\psi_{\bullet \tilde{P}_j}^{\tilde{k}_j}[\Gamma_\sP\cap \Gamma_{\tilde{\sP}}]$ if the two strata intersect transversely, and it is otherwise corrected by the euler class of the obstruction bundle that computes the product $[\Gamma_\sP][\Gamma_{\tilde{\sP}}]$ as a cycle pushed forward from the inclusion of $\Gamma_\sP\cap \Gamma_{\tilde{\sP}}$ into the ambient space. For this reason the expression $\left(\prod_{j=1}^s\psi_{\bullet P_j}^{k_j}[\Gamma_\sP]\right) \left(\prod_{j=1}^{\tilde{s}}\psi_{\bullet \tilde{P}_j}^{\tilde{k}_j}[\Gamma_{\tilde{\sP}}]\right)$
is often formally manipulated as a product, and rewritten as\\ $\prod_{j=1}^s\psi_{\bullet P_j}^{k_j}\prod_{j=1}^{\tilde{s}}\psi_{\bullet \tilde{P}_j}^{\tilde{k}_j}[\Gamma_\sP][\Gamma_{\tilde{\sP}}]$.
\end{remark}

\begin{defi}\label{defdeco} 
Consider a  $\sP$-graph as illustrated in Figure (\ref{fig:image13a}). For this $\sP$-graph, define the following $ \psi $-function :
\begin{equation}\label{psiFn1stThm}
\phi_{\sP}(\psi)=   \psi_{S_{1}}^{k_{S_{1}}}\ldots ... \psi_{S_{q}}^{k_{S_{q}}} \left(\psi_{\bullet P_1}^{k_{P_{11}}+k_{P_{12}}-1}\ldots  \psi_{\bullet{P_l}}^{k_{P_{l1}}+k_{P_{l2}}-1}\right)  .
\end{equation}
\end{defi}

\begin{figure}[htb]
  \includegraphics[width=3in]{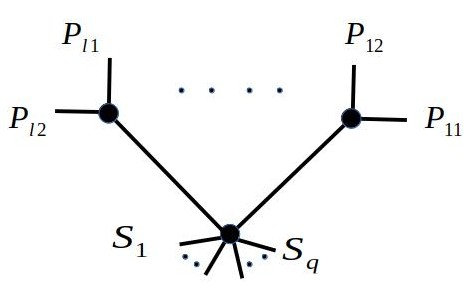}
  \caption{A $\sP$-graph with $l$ forks.}
  \label{fig:image13a}
\end{figure}
\begin{lemma}
On $\overline{M}_{0,n}$, let $D_i$ be a divisor $D(\{i,j\})$ for some $j \neq i$. Then any non-trivial $\prod\limits_{i=1}^{s} D_{i}$, with $s \leq n-3$, is supported on a $\sP$-graph $\Gamma_{\sP}$. 
\label{Dlemma}
\end{lemma}
\begin{proof} We prove by induction.
$D_1$ = 
$\begin{pmatrix}
\includegraphics[width=0.75in]{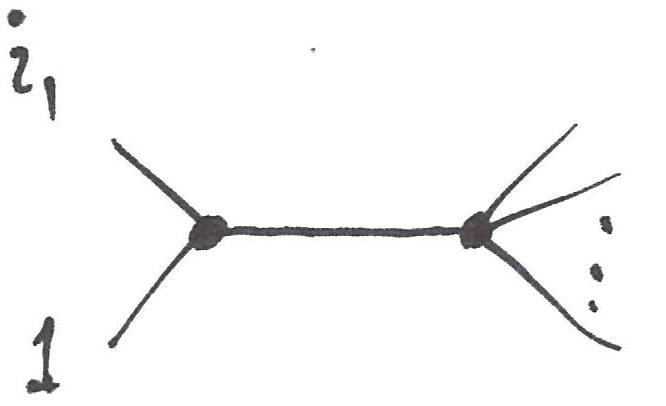}
\end{pmatrix}
$ is clearly a $\sP$-graph, hence trivially supported on a $\sP$-graph. Now,
\begin{align*}
D_1D_2 = \begin{pmatrix}\includegraphics[width=0.75in]{fig5.jpeg}\end{pmatrix}
                \begin{pmatrix}\includegraphics[width=0.75in]{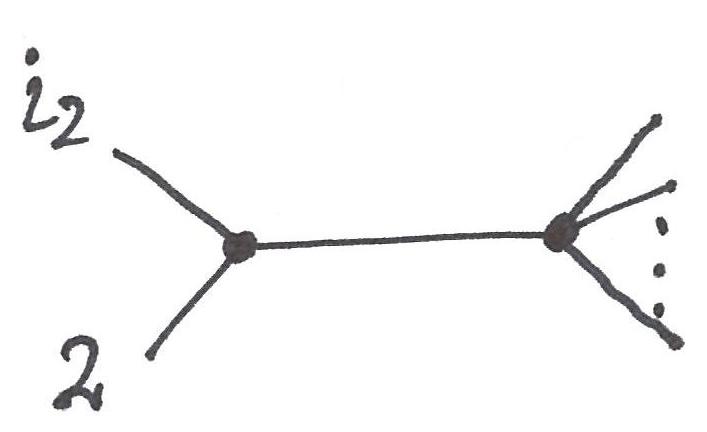}\end{pmatrix}\\
\end{align*}
\begin{align*}
= 
  \begin{cases}
    0, 
                        & \text{if } \vert\lbrace1,i_1\rbrace\cap\lbrace2,i_2\rbrace\vert=1 \\
    \begin{pmatrix}\includegraphics[width=1.3in]{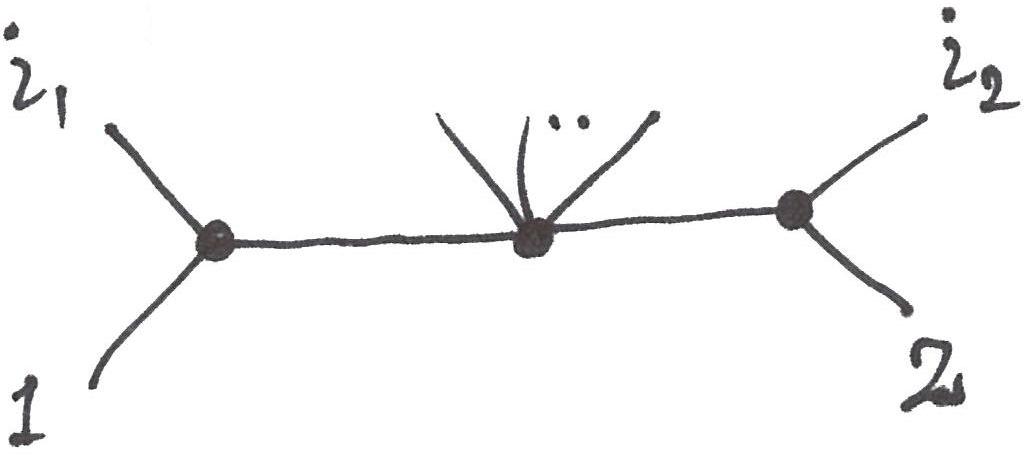}\end{pmatrix},       
                        & \text{if } \vert\{1,i_1\}\cap\{2,i_2\}\vert=\emptyset \\
    - \begin{pmatrix}\includegraphics[width=.85in]{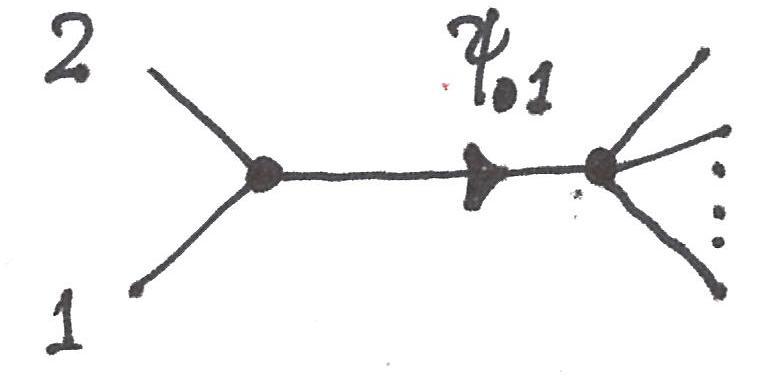}\end{pmatrix}, 
                        & \text{if } \{1,i_1\}=\{2,i_2\}
  \end{cases}
\end{align*}
The first is a trivial intersection, the second is a transverse intersection, and the third is a non-transverse intersection. In the case of non-transverse intersection, we decorate the graph with a $\psi$-class which here we denote by $\psi_{\bullet 1}$, which corresponds to $\psi_{\bullet P_1}$ as in Definition \ref{defpsibullet}. 
So, all the non-trivial intersections $D_1D_2$ are supported on $\sP$-graphs. Now suppose $\prod\limits_{i=1}^{s} D_{i}$ is supported on a $\sP$-graph, and suppose that this $\sP$-graph has $j$ number of forks. Then,  
$$(\prod\limits_{i=1}^{s} D_{i})D_{s+1}$$ is supported on 
$$
 \begin{pmatrix}\includegraphics[width=1.4in]{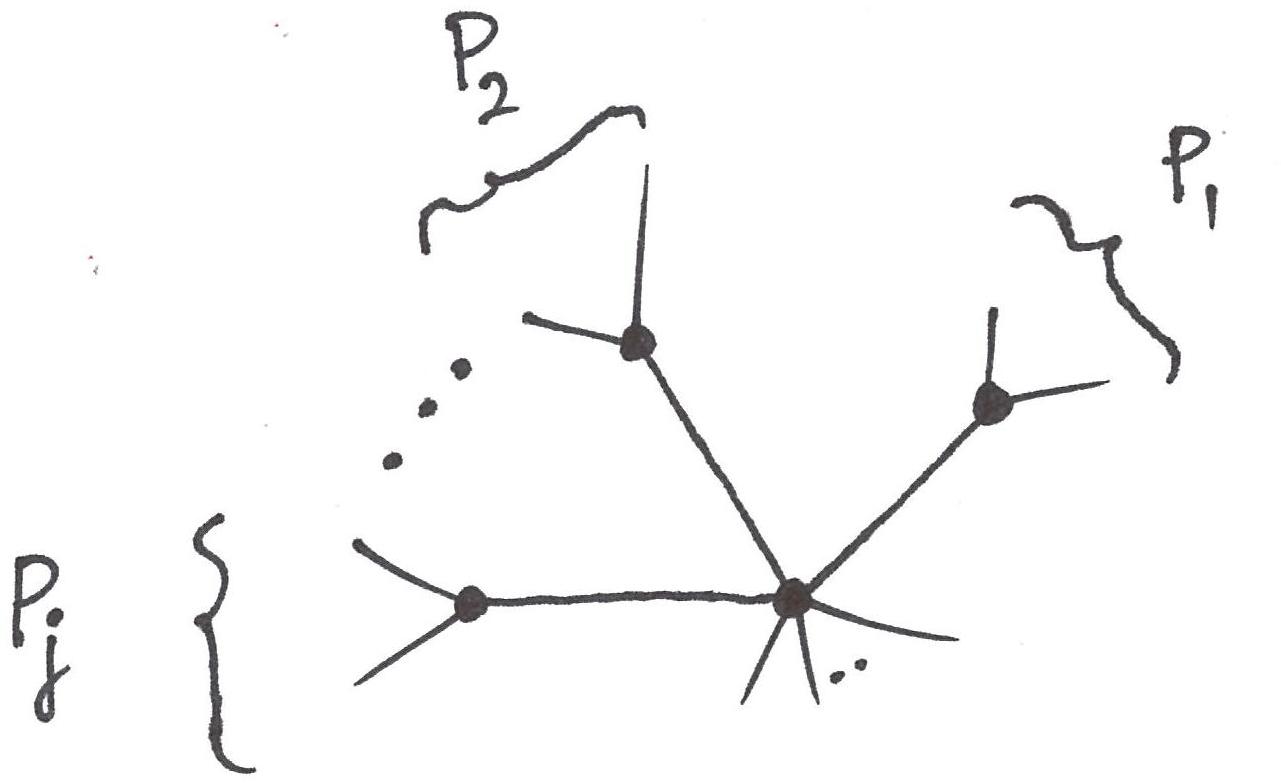}\end{pmatrix} \begin{pmatrix}\includegraphics[width=0.8in]{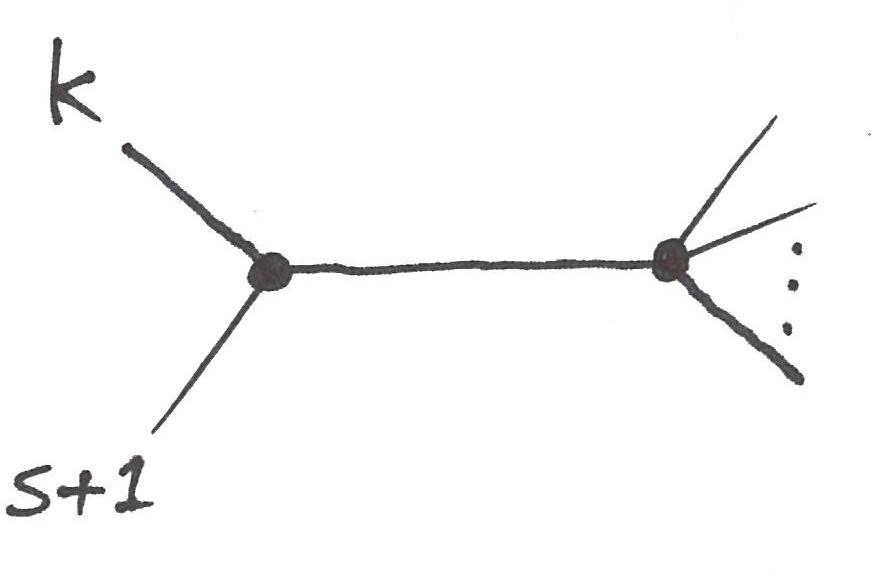}\end{pmatrix}
$$
$$
=
  \begin{cases}
                0,              & \text{if} \{s+1,k\}\nsubseteq P_i, S \hspace{0.1cm}\forall \hspace{0.1cm} i \\ 
                                \begin{pmatrix}\includegraphics[width=1.3in]{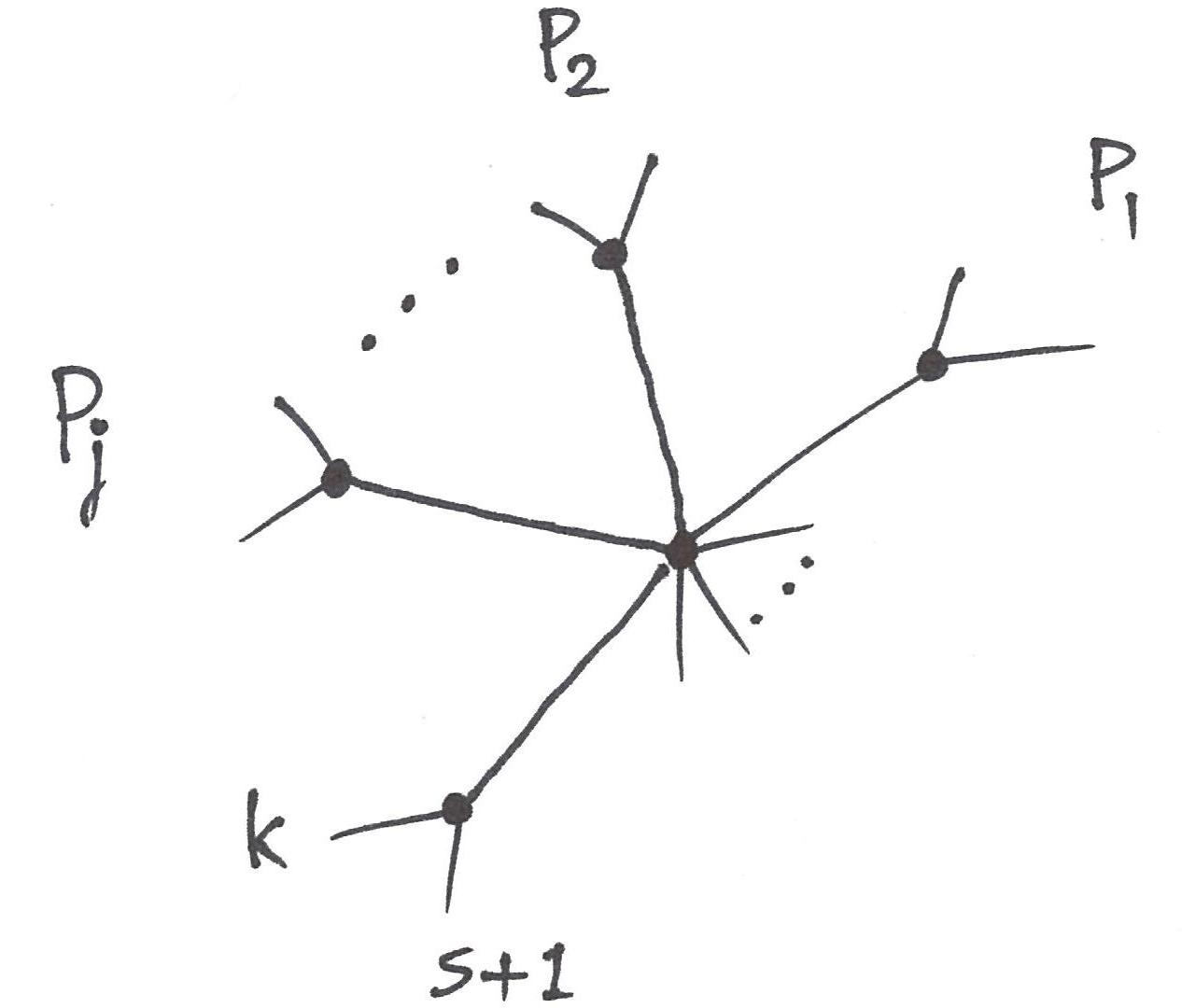}\end{pmatrix}, 
                                                & \text{if} \{s+1,k\}\subseteq\ S \\
                - \begin{pmatrix}\includegraphics[width=1.5in]{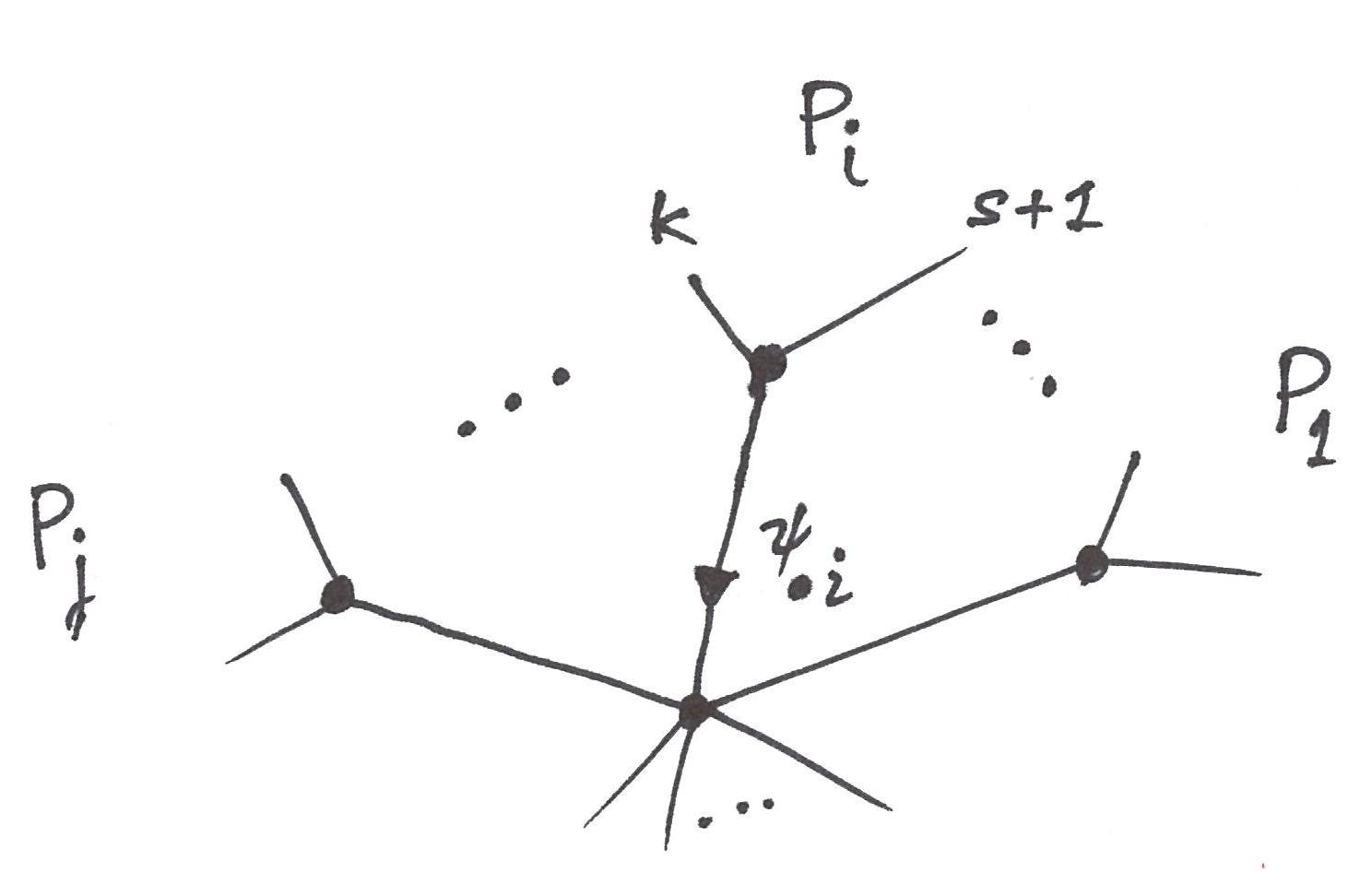}\end{pmatrix},
                                & \text{if}\{s+1,k\}\subseteq\ P_i \hspace{0.1cm} \text{for some} \hspace{0.1cm} i \\
  \end{cases}
$$
Again here, the first is a trivial intersection, the second is a transverse intersection, and the third is a non-transverse intersection; and in the case of non-transverse intersection, we decorate the graph with a $\psi$-class which here we denote by $\psi_{\bullet i}$, corresponding to subset $P_i$ in the partition of $[n]$ which corresponds to $\psi_{\bullet P_i}$ as in Definition \ref{defpsibullet}.
So, all graphs that we get for non-trivial intersections are in fact $\sP$-graphs.

\end{proof}

The following definition is a special case of Definition \ref{defpsibullet}, that will be of particular use in Lemma \ref{lem1}.
\begin{defi}\label{psiOnD}
Given a divisor $D(\{i,j\})$ on $\overline{M}_{0,n} $, consider the following commutative diagram:
$$
\xymatrixrowsep{0.5in}
\xymatrixcolsep{0.6in}
\xymatrix{\overline{M}_{0,n} \\D(\{i,j\}) \ar[u]^{gl} \ar[r]^f & \overline{M}_{0, [n]\setminus \{i,j\}   \bigcup \bullet \{i,j\}}
}
$$
where $gl$ is a gluing morphism, and $f$ is a projection map to $\overline{M}_{0,n-1} $ with marked points $ [n]\setminus \{i,j\}   \bigcup \bullet \{i,j\} $.
Then,
$$
\psi_{\bullet {\left\lbrace i,j\right\rbrace }}^{k} D(\{i,j\}) := gl_{\ast} f^{\ast} \psi_{\bullet {\left\lbrace i,j\right\rbrace }}^{k}.
$$
\end{defi}

\begin{lemma}
\label{lem1}
We use notation as in Definitions \ref{defpsibullet}, \ref{psiOnD} and intersection theoretic calculus as explained in Remark \ref{remcal}; on $ \overline{M}_{0,n} $, we have
\begin{equation}\label{psibulletExpansion}
 \hat \psi_i^{k_i} =  \psi_i^{k_i}- \sum_{j} \psi_{\bullet {\left\lbrace i,j\right\rbrace }}^{k_i-1} D(\{i,j\})
\end{equation}
\end{lemma}
\begin{proof}
Using Corollary \ref{psihatpsiCor},
$$
 \hat \psi_i^{k_i} =  (\psi_i- \sum_{j} D(\{i,j\}))^{k_i}
$$
$$
=   \psi_i^{k_i} + \ldots + (-1)^{k_i} \left( \sum_{j} D(\{i,j\})\right) ^{k_i}
$$
$$
=   \psi_i^{k_i} + (-1)^{k_i} \sum_{j} D(\{i,j\})^{k_i}
$$
$$
=  \psi_i^{k_i} + \ldots + (-1)^{k_i} (-1)^{k_i-1} \psi_{\bullet {\left\lbrace i,j\right\rbrace }}^{k_i-1} \sum_{j} D(\{i,j\})
$$
$$
=  \psi_i^{k_i}- \sum_{j} \psi_{\bullet {\left\lbrace i,j\right\rbrace }}^{k_i-1} D(\{i,j\})
$$
The third equality happens as all the terms in the expansion except the first and the last vanish due to dimension reasons. Furthermore, it's only the self intersections that are non-zero in the last term in the expansion of $  (\psi_i- \sum_{j} D(\{i,j\}))^{k_i} $. The fourth and fifth equalities follow from the $k_i$ self intersections of $ \sum_{j} D(\{i,j\}) $. 
\end{proof}
For the result in the above lemma, we will use the following notation for brevity. 
\begin{equation}\label{psibulletExpansion1}
 \hat \psi_i^{k_i} =  \psi_i^{k_i}- \psi_{\bullet_i}^{k_i-1} D_i
\end{equation}
where 
$$
D_i := \sum_{j}D(\{i,j\})
$$
and 
$$
\psi_{\bullet_i}^{k_i-1} D_i := \sum_{j} \psi_{\bullet {\left\lbrace i,j\right\rbrace }}^{k_i-1} D(\{i,j\})
$$
Also, again for brevity, if the $\sP$-graph in figure (\ref{fig:image13a}) corresponds to partition $\sP = \{P_1,P_2,..,P_s,S_1,S_2,...,S_q   \}$, then we will denote $\psi_{\bullet P_j}$ also as $\psi_{P_j}$ in what follows.
\begin{thm}\label{arbintThm}
With $\sP, F, S, P_j,P_{j1}, P_{j2}, S_i$ as defined above, for $n\geq 5$ we have for $ \overline{M}_{0,n} $:
\begin{equation}
\hat \psi_1^{k_1}\hat \psi_2^{k_2} \ldots \hat \psi_n^{k_n} =
\sum_{\sP \in \mathfrak{P}} (-1)^{|F|} \left[  \Gamma_{{\sP}} \right] \prod_{S_i \in S} \psi_{S_i}^{k_{S_i}} \prod_{P_j \in F} \psi_{P_j}^{k_{P_{j1}} + k_{P_{j2}} - 1} 
\label{arbintThmEq}
\end{equation}
where $\left[  \Gamma_{{\sP}} \right]$ is the class of boundary stratum in $ \overline{M}_{0,n} $ with dual graph $ \Gamma_{{\sP}}$.
\end{thm}
\begin{proof}
Let $D_i = \sum_{j}D(\{i,j\})$. Omitting $\hat \psi$'s with $0$-exponents, and assuming, without loss of generality, that the first $r$ number of $\hat \psi$'s remain with nonzero-exponents,  
$ \hat \psi_1^{k_1} \hat \psi_2^{k_2} \ldots \hat \psi_n^{k_n} 
= \hat \psi_1^{k_1}\hat \psi_2^{k_2}\hat \psi_3^{k_3}...\hat \psi_r^{k_r}$
. Then, 
$$\hat \psi_1^{k_1}\hat \psi_2^{k_2}\hat \psi_3^{k_3}...\hat \psi_r^{k_r} =  (\psi_1- D_1)^{k_1}(\psi_1- D_2)^{k_2}\ldots (\psi_r-  D_r)^{k_r} $$ using relation (\ref{psihatpsiCorEq})
$$= (\psi_1^{k_1}- \psi_{\bullet_1}^{k_1-1} D_1)
        (\psi_2^{k_2}-  \psi_{\bullet_2}^{k_2-1} D_2) \ldots
        (\psi_r^{k_r}-  \psi_{\bullet_r}^{k_r-1} D_r)$$
        using relation (\ref{psibulletExpansion1}) 
\begin{equation}\label{arbintThmEq2}
=(-1)^s \sum_{s=0}^{r} \psi_{\bullet_{i_1}}^{k_{i_1}-1} \psi_{\bullet_{i_2}}^{k_{i_2}-1}\ldots\psi_{\bullet_{i_s}}^{k_{i_s}-1} D_{i_1} D_{i_2}\ldots D_{i_s} \psi_{i_{s+1}}^{k_{i_{s+1}}} \ldots \psi_{i_r}^{k_{i_r}}  
\end{equation}
with $1 \leq i_j \leq r $ and $i_1 < i_2 < \ldots < i_s $. Now, each term in the expansion of the expression above is supported on a $\sP$-graph from Lemma \ref{Dlemma}. 
Pick a $\sP$-graph with $s$ number of colored half-edges on forks and form the corresponding $\Gamma_{\sP}^d$. There are two possibilities: 1) at least one fork has both half-edges uncolored, or 2) all forks of $\Gamma_{\sP}^d$ have at least one colored half-edge. Let the second type of graph (figure \ref{fig:image13}) have $t$ number of forks with both half-edges colored as shown below. Here, $\sP =\left\lbrace  P_1, P_2, \ldots, P_{s-\frac{t}{2}}, S_1,\ldots, S_{r-s}\right\rbrace = \left\lbrace  \{i_1, i_2\}, \ldots, \{i_s,j_{s-t}\}, i_{s+1},\ldots, i_{r}\right\rbrace.$ 
\begin{figure}[!htb]
\centering
  \includegraphics[width=4in]{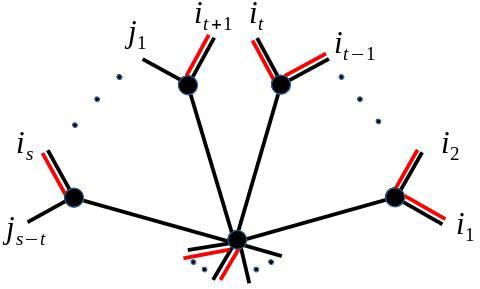}
        \caption{A decorated $\sP$-graph with no uncolored fork}
  \label{fig:image13}
\end{figure}
Intersect both types of $\sP$-graph with $ \phi_{\sP}(\psi) $ as defined in (\ref{psiFn1stThm}), which for this $\sP$-graph is :

\begin{equation}
\phi_{\sP}(\psi)= \left(\psi_{\bullet 1 }^{k_{i_{1}}+k_{i_{2}}-1}\ldots  \psi_{\bullet{{\frac{t}{2}}}}^{k_{i_{t-1}}+k_{i_{t}}-1} \psi_{\bullet_{{\frac{t}{2}+1}}}^{k_{i_{t+1}}-1} \ldots \psi_{\bullet{{s-\frac{t}{2}}}}^{k_{i_{s}}-1}\right) \psi_{i_{s+1}}^{k_{i_{s+1}}}\ldots ... \psi_{i_r}^{k_{i_r}} 
\end{equation}
\textit{Claim} : The intersection of $ \phi_{\sP}(\psi) $ with the first type of graphs gives $0$.\\
\textit{Proof} Without loss of generality, suppose the first type of $\sP$-graph have $ i_1 $ and $ i_2 $ on a fork with $ k_{i_1}=k_{i_2} = 0 $. Then $k_{i_1}+k_{i_2}-1= -1$ and 
$$
\phi_{\sP}(\psi)= \psi_{\bullet_1}^{-1}  \ldots  \psi_{\bullet{{\frac{t}{2}}}}^{k_{i_{t-1}}+k_{i_{t}}-1} \psi_{\bullet_{{\frac{t}{2}+1}}}^{k_{i_{t+1}}-1} \ldots \psi_{\bullet{{s-\frac{t}{2}}}}^{k_{i_{s}}-1} \psi_{i_{s+1}}^{k_{i_{s+1}}}\ldots ... \psi_{i_r}^{k_{i_r}}
$$
is $0$ as negative power of a $\psi$ class by standard convention is $0$. 
Now, consider the second type of  $\sP$-graph (Figure \ref{fig:image13}).\\
\textit{Claim}: $\phi_{\sP} \left( \psi \right) \cdot\Gamma_{{\sP}}^d$, where $\Gamma_{{\sP}}^d$ is the second type of graph, uniquely determines a term in the expansion (\ref{arbintThmEq2}) above.\\ 
\textit{Proof}: 
Define a map from the set of $\sP$-graphs to the terms in the expansion (\ref{arbintThmEq2}) as follows. For each $ i_l $ in $ \Gamma_{{\sP}}^d $ where $ i_l $ is a colored half-edge on a fork, assign a $ D_{i_l} $, to form the product $\prod_{l=1}^{s} D_{i_l}$. Now, for each fork on the graph with half-edges $ i_{l} $ and $ i_{q} $, assign $ \psi_{\bullet_m}^{k_{i_{l}}+k_{i_{q}}-1} $ and form their product; the result is $ \phi_{\sP}(\psi) $. Then $\phi(\psi)\cdot \prod_{l=1}^{s} D_{i_l}$ is precisely the term in the expansion (\ref{arbintThmEq2}) that $\Gamma_{{\sP}}^d$ maps to. Further, reversing the process, we get the preimage of a term in expansion (\ref{arbintThmEq2}) the unique $\sP$-graph (figure \ref{fig:image13}). So, the map is in fact a bijection. Now, \\
$$
\phi_{\sP} \left( \psi \right)\cdot \left[\Gamma_{{\sP}}^d\right] =
$$
\begin{equation}
\\ (-1)^{s+\frac{t}{2}} \left[ \Gamma_{{\sP}} \right] \left(\psi_{\bullet 1 }^{k_{i_{1}}+k_{i_{2}}-1}\ldots  \psi_{\bullet{{\frac{t}{2}}}}^{k_{i_{t-1}}+k_{i_{t}}-1} \psi_{\bullet_{{\frac{t}{2}+1}}}^{k_{i_{t+1}}-1} \ldots \psi_{\bullet{{s-\frac{t}{2}}}}^{k_{i_{s}}-1}\right) \psi_{i_{s+1}}^{k_{i_{s+1}}}\ldots ... \psi_{i_r}^{k_{i_r}}
\end{equation}\label{intEval1}
$$= (-1)^{|F|} \left[  \Gamma_{{\sP}} \right]  \prod_{S_i \in S} \psi_{S_i}^{k_{S_i}} \prod_{P_j \in F} \psi_{P_j}^{k_{P_{j1}} + k_{P_{j2}} - 1} .$$
\end{proof}
\section{Numerical intersections}\label{numer}
In this section, we develop two corollaries of theorem (\ref{arbintThm}) for two versions of a closed formula for top intersections of $\hat \psi$-classes on $\overline{M}_{0,n  } $.Then we encode this formula in a generating function obtained by applying a differential operator to the Witten potential. As pointed out earlier, these corollaries (\ref{mainThm} and \ref{mainThm1}) can also be deduced from theorem $7.9$ in \cite{alexeev_guy_2008}. For our work, we develop specific and explicit closed formulas here and base our combinatorial analysis closely on the structure of dual graphs.   
\begin{cor}\label{mainThm}
With $\sP, F, P_j, P_{j1}, P_{j2}, S_i$ as defined above, for $n\geq 5$ we have:
\begin{equation}
\int_{ \overline{M}_{0,n}} \hat \psi_1^{k_1} \hat \psi_2^{k_2} \ldots \hat \psi_n^{k_n} =
\sum_{\sP \in \mathfrak{P}} (-1)^{|F|}   {{|\sP|-3}\choose{ \langle {k_{P_{i1}}+k_{P_{i2}}-1} \rangle,\langle k_{S}\rangle}}
\label{mainthmEq}
\end{equation}
where $\sum k_i = n-3,$ and 
for a $\sP =\left\lbrace  P_1, P_2, \ldots, P_s, S_1,\ldots, S_q\right\rbrace, \\ \langle {k_{P_{i1}}+k_{P_{i2}}-1}\rangle= k_{P_{11}}+k_{P_{12}}-1, \ldots,k_{P_{s1}}+k_{P_{s2}}-1,$ and $ \langle k_{S}  \rangle = k_{S_1},\ldots,k_{S_q}$.
\end{cor}
\begin{proof}
The proof of this is same as for theorem (\ref{arbintThm}) except in the last part of evaluation of $ \phi_{\sP} \left( \psi \right)\cdot \left[  \Gamma_{{\sP}} \right] $. Here when $\sum k_i = n-3$, this evaluation gives\\  
$$
\phi_{\sP} \left( \psi \right)\cdot \left[  \Gamma_{{\sP}} \right] 
$$
$$=(-1)^{s+\frac{t}{2}} {{n-3-(s-\frac{t}{2})}\choose{k_{i_1}+k_{i_2}-1, \ldots,k_{i_{t+1}}-1,\ldots,k_{i_{s}}-1,k_{i_{s+1}},\ldots, k_{i_r}}}$$
$$=(-1)^{|F|} {{n-3-|F|}\choose{k_{i_1}+k_{i_2}-1, \ldots,k_{i_{s}}+k_{j_s}-1,k_{i_{s+1}},\ldots, k_{i_r}}}$$
as $k_{i_{r+1}}=\ldots= k_{i_{n}}=0$
$$= (-1)^{|F|}   {{|\sP|-3}\choose{ \langle {k_{P_{i1}}+k_{P_{i2}}-1} \rangle,\langle k_{S}\rangle}}. $$
\end{proof}
We can reduce the complexity of the computation of ($\ref{mainthmEq}$) if we can remove from $\mathfrak{P}$ those partitions $\mathcal{P}$'s whose graphs evaluate to $0$ when intersected with $\phi(\psi)$.
Also, we can collect together terms that come from permuting the marked points $j_1,j_2,...,j_{s-t}$ in the partition $\mathfrak{P}$ as all these terms evaluate to the same value as $k_{j_i}=0$ for all these $j_i$.

Form a new set $\mathfrak{P}'$ in the following way: 
Make the powerset $\mathfrak{R}$ of $[r]$, where $r$ denotes the number of $\psi$'s with non-zero exponent in the $\psi$-monomial. For each set $\sR \in \mathfrak{R}$, form all subsets $\mathcal{P}'$ of $\sR$ whose elements have cardinality 2 or 1 with the upper bound of number of subsets of cardinality 2 fixed at $\floor{\frac{n}{2}}$. Call $\mathfrak{P}'$ the set of all $\mathcal{P}'$.
This set $\mathfrak{P}'$ can also be obtained from $\mathfrak{P}$ via the following map: 
Given a partition $\mathcal{P}$, project to a $\mathcal{P}'$ by forgetting all points $S_i \in \mathcal{P}$ and in a $P_j \in \mathcal{P}$ forget a point $P_{ji}$ if $k_{P_{ji}}=0$. More formally, $\mathcal{P}'= \{P_1',P_2',..,P_i',S_1',S_2',..S_j'\}$ where $P_i' = P_i$ if $k_{Pi1}>0$ and $k_{Pi2}>0$; $S_i' = P_i \setminus \{P_{ij}\}$ if $k_{P_{ij}}=0$. This is an onto map. Each $P_j^{'}$ has cardinality 2, and each $S_i{'}$ has cardinality 1. Denote by $\mathcal{P}'^c$ the set $[r]\setminus P'_1 \cup ...\cup P'_i\cup S_1'\cup S_2' \cup..\cup S_j'$.
\begin{cor}\label{mainThm1}
With $\mathfrak{P}'$ as defined above, for $n\geq 5$ we have:
\begin{equation}
\begin{multlined}
\int_{ \overline{M}_{0,n  }} \hat \psi_1^{k_1}\hat \psi_2^{k_2}\hat \psi_3^{k_3}...\hat \psi_r^{k_r} =\\
\sum_{\mathcal{P}' \in \mathfrak{P}'} (-1)^{s+\frac{t}{2}} \frac{(n-r)!}{(n-r-s+t)!} {{n-3-(s-\frac{t}{2})}\choose{k_{i_1}+k_{i_2}-1,..,k_{i_{t+1}}-1,..,k_{i_{s}}-1,k_{{c_1}},.., k_{c_q}}}
\end{multlined}
\end{equation}
where $t= 2|\{ P_i^{'} \}|$, $s= |\{ S_i{'}  \}| + t$ and $\{{c_1},..,{c_q} \}=\mathcal{P}'^c$
 
\end{cor}
\begin{proof}
Corresponding to a partition $\mathcal{P}$, form a corresponding decorated graph for $\mathcal{P}^{'}$ by uncoloring any half-edges on the central node, and `forgetting' the $j_i$'s on uncolored half-edges on the forks as discussed above. This corresponds to $\mathcal{P}' = \lbrace P_1^{'}, P_2^{'}, \ldots, P_{t/2}^{'}, S_1^{'},\ldots, S_{s-t}^{'} \rbrace$, where $P_i^{'}$'s correspond to nodes with both half-edges colored, and $S'_j = i_{t+j}$ in the decorated graph of $\mathcal{P}$ in figure (\ref{fig:image13}). Now consider intersection of this graph with $\phi_{\sP}(\psi)$ as defined earlier in (\ref{psiFn1stThm}) :\\ 
$\phi_{\sP}(\psi) \cdot \begin{pmatrix}
\includegraphics[width=2in]{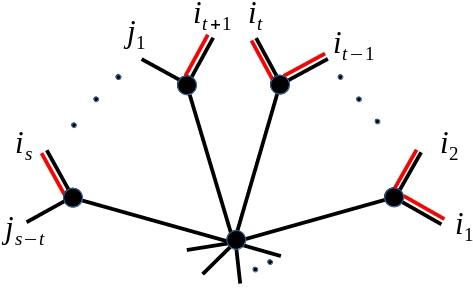}
\end{pmatrix}
$
$$
\noindent \cong  \overline{M}_{0,3} \times \overline{M}_{0,3},\ldots,\overline{M}_{0,3} \times \phi(\psi). \overline{M}_{0,n-3-(s-\frac{t}{2})}
$$
$$
=(-1)^{s+\frac{t}{2}} {{n-3-(s-\frac{t}{2})}\choose{k_{i_1}+k_{i_2}-1, ..,k_{i_{t-1}}+k_{i_t}-1,k_{i_{t+1}}-1,..,k_{i_{s}}-1,k_{{c_1}},.., k_{c_q}}}
$$
As there are $\frac{(n-r)!}{(n-r-s+t)!}$ ways of choosing the uncolored half-edges on the forks, corresponding to $j_i$'s, the term evaluates to 
$$
=(-1)^{s+\frac{t}{2}} \frac{(n-r)!}{(n-r-s+t)!} {{n-3-(s-\frac{t}{2})}\choose{k_{i_1}+k_{i_2}-1, ..,k_{i_{t+1}}-1,..,k_{i_{s}}-1,k_{{c_1}},.., k_{c_q}}}
$$
\end{proof}
\subsection{Generating Function for the top intersections}\label{secGenfn}
We start with the generating function- Witten potential(\cite{kock})- also called, in topological gravity, \textit{the total free energy}. The correlation functions are defined as intersection numbers on the moduli space of stable n-pointed curves (here for genus $0$) as
 
$$
\langle \tau_{k_1} \ldots \tau_{k_n} \rangle :=\int_{ \overline{M}_{0,n}}  \psi_1^{k_1} \psi_2^{k_2} \ldots  \psi_n^{k_n} 
$$   
Collecting all tau's with equal exponent, we can write 
$ \langle \tau_{k_1} \ldots \tau_{k_n} \rangle = \\
\langle \tau_{0}^{s_0} \tau_{1}^{s_1} \tau_{2}^{s_2} \ldots \tau_{m}^{s_m} \rangle $. Now, define $ \textbf{s} = \left(  s_0, s_1, \ldots \right)  $, \\ and 
$\langle \mathbf{{\tau}^{s}} \rangle := \langle \tau_{0}^{s_0} \tau_{1}^{s_1} \tau_{2}^{s_2} \ldots \tau_{m}^{s_m} \rangle $. So, for each sequence \textbf{s}, there is a correlation function $\langle {\mathbf{\tau}}^{\mathbf{s}} \rangle$; and $ \vert \mathbf{s} \vert := \sum{s_i} $ is the number of marks $n$. For the generating function, all these correlation functions are collected and used as coefficients in a formal power series. Using notation $ {\mathbf{t}}^{\mathbf{s}} = \displaystyle \prod_{i=0}^{\infty} t_{i}^{s_i} $, and $ \mathbf{s!} = \displaystyle \prod_{i=0}^{\infty} {s_{i}!}  $, the generating function is
\begin{equation}\label{WittenPot}
F(\mathbf{t}) := \sum_{\mathbf{s}} \frac{{\mathbf{t}}^{\mathbf{s}}}{\mathbf{s!}} \langle {\mathbf{\tau}}^{\mathbf{s}} \rangle 
\end{equation}
$$  
= (1) \frac{t_0^3t_1}{3!} 
+ (1) \frac{t_0^4t_2}{4!} + (2) \frac{t_0^3t_1^2}{3!2!}
+ (1) \frac{t_0^5t_3}{5!}+ (3) \frac{t_0^4t_1t_2}{4!}+ (6) \frac{t_0^3t_1^3}{3!3!} +\ldots\ldots$$\\
where the coefficients of appropriate terms give the intersection numbers $\int_{ \overline{M}_{0,n}}  \psi_1^{k_1} \psi_2^{k_2} \ldots  \psi_n^{k_n}$. Observe that  the total codimension of the integrand in $\langle {\mathbf{\tau}}^{\mathbf{s}}\rangle$ is $ \sum is_i $, so $  |s| - 3 = \sum i s_i$. With this generating function, the String equation (\ref{stringEq}) for $\overline{M}_{0,n}$ is encoded in the differential equation -\\
$$\frac{\partial}{\partial t_0} F = \frac{t_{0}^{2}}{2} +  \sum_{i=0}^{\infty} t_{i+1} \frac{\partial}{\partial t_i} F$$\\
\begin{defi}\label{GenPot}
Define a new generating function 
\begin{equation}\label{GenPotEq}
G(\mathbf{t}) := \sum_{\mathbf{s}} \frac{{\mathbf{t}}^{\mathbf{s}}}{\mathbf{s!}} \langle {\hat{\mathbf{\tau}}}^{\mathbf{s}} \rangle   
\end{equation}
where $ \langle {\hat{\mathbf{\tau}}}^{\mathbf{s}} \rangle = \int_{ \overline{M}_{0,n}} \hat \psi_1^{k_1}\hat \psi_2^{k_2} \ldots \hat \psi_n^{k_n}$, and $\sum k_i = n-3$. 
\end{defi}
So, $G(\mathbf{t})$ has as coefficients of the monomials $\frac{{\mathbf{t}}^{\mathbf{s}}}{\mathbf{s!}}$ the intersection numbers $\int_{ \overline{M}_{0,n}} \hat \psi_1^{k_1}\hat \psi_2^{k_2} \ldots \hat \psi_n^{k_n}$ for any value of $n$ and any values of $k_i$'s, with $\sum k_i = n-3$.
\begin{defi}
Given two operators $\mathcal{J}_1 = t_{p}t_{q} \frac{\partial}{\partial t_{r}} $, and $\mathcal{J}_2 = t_{p'}t_{q'} \frac{\partial}{\partial t_{r'}} $, 
the normal ordering of the product, denoted $ : \mathcal{J}_1 \mathcal{J}_2 : $, means that we treat the $t_i$'s and $ \frac{\partial}{\partial t_{{i}}} $'s as commuting variables, and bring all $t_i$'s to the left of $ \frac{\partial}{\partial t_{{i}}} $'s. E.g., here 
$$
\mathcal{J}_1 \mathcal{J}_2 \hspace{0.15cm} = t_{p}t_{q} \frac{\partial}{\partial t_{r}} t_{p'}t_{q'} \frac{\partial}{\partial t_{r'}},
$$ but
$$
: \mathcal{J}_1 \mathcal{J}_2 :\hspace{0.15cm} = t_{p}t_{q} t_{p'}t_{q'} \frac{\partial}{\partial t_{r}} \frac{\partial}{\partial t_{r'}}.
$$
Furthermore, 
$$
: \mathcal{J}_1  +  \mathcal{J}_2 :\hspace{0.15cm} = \hspace{0.15cm} : \mathcal{J}_1 : + : \mathcal{J}_2 :
$$
\end{defi}
It follows from the above definition that normal ordering extends to arbitrary products of operators, and extends by linearity to arbitrary sums (including infinite sums) of operators. In particular, it extends to $\hspace{0.1cm} : e^{\mathcal{J}} :$. 
 
\begin{thm}{\label{generatingFnthm}}
With $G(\textbf{t})$ as defined above, 
\begin{equation}\label{GenPotexp}
G(\textbf{t}) = \hat \sL (F(\textbf{t})) - \frac{t_0^3}{3!} + 2\frac{t_0^3 t_1}{3!} 
\end{equation}
where $\hat \sL = \hspace{0.1cm} : e^{- \sL} :$, and $\hspace{0.1cm} : e^{- \sL} :$ denotes the operator with normal ordering, and 
\begin{equation}\label{LInOp}
\sL = \frac{1}{2} \sum_{i,j=0}^{\infty} t_{i}t_{j} \frac{\partial}{\partial t_{{i}+{j}-1}} \end{equation}
\end{thm}
Before we prove the above theorem, consider the $\hat{\psi}$-monomial $\int_{ \overline{M}_{0,n}} \hat \psi_1^{n-3}$. It is the coefficient of $t$-monomial $\frac{t_0^{(n-1)}t_{n-3}}{(n-1)!}$ in $G(\textbf{t})$. When we apply the operator $\hat \sL = 1 - \sL + : \frac{{\sL}^2}{2!}: - \ldots$ to $F(\textbf{t})$, only the  terms with the following $t$-monomials in $F(\textbf{t})$ contribute to the term with $t$-monomial $\frac{t_0^{(n-1)}t_{n-3}}{(n-1)!}$ in $ \hat \sL (F(\textbf{t})) $:
$\frac{t_0^{(n-1)}t_{n-3}}{(n-1)!}$ and $\frac{t_0^{(n-2)}t_{n-4}}{(n-2)!}$.
The corresponding operations happen as follows. 
The first term of $ \hat \sL$, which is $1$, acts on $F(\textbf{t})$ to produce the $t$-monomial $\frac{t_0^{(n-1)}t_{n-3}}{(n-1)!}$ as it is. Then, among the summands in $- \sL$, the only operator that produces the $t$-monomial $\frac{t_0^{(n-1)}t_{n-3}}{(n-1)!}$ is $-t_0t_{n-3} \frac{\partial}{\partial t_{n-4}}$, which acts as follows. 
$$
-t_0t_{n-3} \frac{\partial}{\partial t_{n-4}} \left(  \frac{t_0^{(n-2)}t_{n-4}}{(n-2)!} \right)  = -(n-1) \frac{t_0^{(n-1)}t_{n-3}}{(n-1)!} 
.$$ 

No other term in $F(\textbf{t})$ contributes to the coefficient of the $t-$monomial $\frac{t_0^{(n-1)}t_{n-3}}{(n-1)!}$ in $ \hat \sL(F(\textbf{t}))$. As the monomial $\frac{t_0^{(n-1)}t_{n-3}}{(n-1)!}$ has the coefficient $\int_{ \overline{M}_{0,n}}  \psi_1^{n-3}$, and the monomial $\frac{t_0^{(n-2)}t_{n-4}}{(n-2)!}$ has the coefficient $\int_{ \overline{M}_{0,{n-1}}} \psi_1^{n-4} $ in $F(\textbf{t})$, the coefficient of $\frac{t_0^{(n-1)}t_{n-3}}{(n-1)!}$ in  $ \hat \sL (F(\textbf{t})) $ is \\ $ \int_{ \overline{M}_{0,n}}  \psi_1^{n-3} - (n-1) \int_{ \overline{M}_{0,{n-1}}} \psi_1^{n-4} $, which equals $\int_{ \overline{M}_{0,n}} \hat \psi_1^{n-3}$ from Corollary (\ref{mainThm}).
Observe that both the contributions correspond to the two types of $\mathcal{P}$-graphs that make non-zero contributions to $\int_{\overline{M}_{0,n}} \hat \psi_1^{n-3}$ in corollary (\ref{mainThm}). The first are of the type with no forks, and the second are of the type with one fork. 
\begin{defi}\label{pkDef}
For each $\mathcal{P}$-graph, the corresponding $\mathcal{P}_{k}$-graph is defined by replacing each $ i \in [n] $ on the $\mathcal{P}$-graph by $ k_i $.  
\end{defi}
Clearly the map $ \{\mathcal{P}$-$graphs\} \to \{\mathcal{P}_{k}$-$graphs\} $ is a surjection.
\begin{lemma}{\label{p-pk-Map}}
In genus $0$, for a given $\langle \tau_{k_1} \ldots \tau_{k_n} \rangle =\int_{ \overline{M}_{0,n}}  \psi_1^{k_1} \psi_2^{k_2} \ldots  \psi_n^{k_n} $ as defined above, let 
$ \textbf{s} = \left(  s_0, s_1, \ldots s_n\right)  $ be such that $ \langle \tau_{k_1} \ldots \tau_{k_n} \rangle = 
\langle \tau_{0}^{s_0} \tau_{1}^{s_1} \tau_{2}^{s_2} \ldots \tau_{n}^{s_n} \rangle $. 
Consider a $\mathcal{P}_{k}$-graph with $m$ forks with $q$ distinct $k_i$'s appearing on the forks; let such $k_i$'s be $\left\lbrace  k_{1}, k_{2}, \ldots k_{q} \right\rbrace $. Let $l_{i}$ be the number of times a given $k_i$ appears on any fork on the $\mathcal{P}_{k}$-graph. Then the number of $\mathcal{P}$-graphs that map to this $\mathcal{P}_{k}$-graph is given by:
\begin{equation}\label{noP-Pkgarphs}
 \frac{1}{|Aut(\hat{\mathcal{P}}_{k})|}  \frac{s_{k_1}!}{\left( s_{k_1}-l_{1}\right)! } \frac{s_{k_2}!}{\left( s_{k_2}-l_{2}\right)! } \ldots \frac{s_{k_q}!}{\left( s_{k_q}-l_{q}\right)! } 
\end{equation}
where $ |Aut(\hat{\mathcal{P}}_{k})|$ is the number of automorphisms of the subgraph of the $\mathcal{P}_{k}$-graph obtained by removing half-edges on the central node.
\end{lemma}
\begin{proof}
Consider a $\mathcal{P}_{k}$-graph with $m$ forks such that the number of $k_i$'s appearing on the forks is $q$. Then, if all $n$ half-edges are given ordering, the number of corresponding $ \sP $-graphs would be 
$$
s_{0}! s_{1}! s_{2}! \ldots s_{n}!.
$$
Now, we divide by the permutations of half-edges on the central node to get 
$$
\frac{s_{0}!}{\left( s_{0}-l_{0}\right)! } \frac{s_{1}!}{\left( s_{1}-l_{1}\right)! } \frac{s_{2}!}{\left( s_{2}-l_{2}\right)! } \ldots \frac{s_{n}!}{\left( s_{n}-l_{n}\right)! }.
$$
Now, as only $q$ number of $k_i$'s appear on the forks, $l_i = 0$ for $i > q$, so 
$$
\frac{s_{0}!}{\left( s_{0}-l_{0}\right)! } \frac{s_{1}!}{\left( s_{1}-l_{1}\right)! }  \ldots \frac{s_{n}!}{\left( s_{n}-l_{n}\right)! } = \frac{s_{k_1}!}{\left( s_{k_1}-l_{1}\right)! } \frac{s_{k_2}!}{\left( s_{k_2}-l_{2}\right)! } \ldots \frac{s_{k_q}!}{\left( s_{k_q}-l_{q}\right)! }. 
$$
Further we need to divide by permutations of half-edges on the forks. Let $ j_1, j_2, \ldots, j_f $ be the number of forks with the same set of $k_i$'s on them; and let $d$ be the number of forks with both $k_i$'s same on that fork. Then, we divide by $2^d \left(  j_1! j_2! \ldots j_f!\right)   $ to get 
$$
\left( \frac{1}{2^d \left(  j_1! j_2! \ldots j_f!\right)} \right) \frac{s_{k_1}!}{\left( s_{k_1}-l_{1}\right)! } \frac{s_{k_2}!}{\left( s_{k_2}-l_{2}\right)! } \ldots \frac{s_{k_q}!}{\left( s_{k_q}-l_{q}\right)! }.
$$ 
Observe that the number $ \left(2^d \left(  j_1! j_2! \ldots j_f!\right) \right) $ is the number of automorphisms of the subgraph of $\mathcal{P}_{k}$-graph obtained by removing the half-edges on the central node; denote this subgraph as $ \hat{\mathcal{P}}_{k} $. Then the number of $\mathcal{P}$-graphs that map to this $\mathcal{P}_{k}$-graph can be rewritten as:
$$
 \frac{1}{|Aut(\hat{\mathcal{P}}_{k})|}  \frac{s_{k_1}!}{\left( s_{k_1}-l_{1}\right)! } \frac{s_{k_2}!}{\left( s_{k_2}-l_{2}\right)! } \ldots \frac{s_{k_q}!}{\left( s_{k_q}-l_{q}\right)! }
$$ 
which we denote by $ C_{\sP_{k}} $.
\end{proof}
The reason for organizing $ C_{\sP_{k}} $ as in (\ref{noP-Pkgarphs}) will become clear in the proof of theorem (\ref{generatingFnthm}).
\begin{lemma}
With the definitions and notations above, Corollary (\ref{mainThm}) can be rewritten as :
\begin{equation} \label{pkClosedformula}
\begin{multlined}
\int_{ \overline{M}_{0,n}} \hat \psi_1^{k_1}\hat \psi_2^{k_2} \ldots \hat \psi_n^{k_n} =\\
\sum_{\sP_{k} \in \mathfrak{Q}} (-1)^{m} C_{\sP_{k}} 
 \int_{ \overline{M}_{0,(n-m)}} \psi_{\bullet_1}^{k_{i_{1}}+k_{i_{2}}-1}\ldots \psi_{\bullet_{m}}^{k_{i_{2m-1}}+k_{i_{2m}}-1} \psi_{i_{2m+1}}^{k_{i_{2m+1}}}\ldots ... \psi_{i_r}^{k_{i_r}}
\end{multlined}
\end{equation}
where $m$ is number of forks on the $\mathcal{P}_{k}$-graph and $C_{\sP_{k}} $ is the number of $\mathcal{P}$-graphs that map to this $\mathcal{P}_{k}$-graph, and $ \mathfrak{Q} $ is the set of all $\mathcal{P}_{k}$-graphs.
\end{lemma} 
\begin{proof}
This version of closed formula for $ \int_{ \overline{M}_{0,n}} \hat \psi_1^{k_1}\hat \psi_2^{k_2} \ldots \hat \psi_n^{k_n} $ is just a reorganization of (\ref{mainthmEq}) using $\mathcal{P}_{k}$-graphs instead of $\mathcal{P}$-graphs. As the map $ \{\mathcal{P}$-$graphs\} \to \{\mathcal{P}_{k}$-$graphs\} $ is a surjection, we get all the terms in (\ref{mainthmEq}).
\end{proof}
Now, for a general $\mathcal{P}_{k}$-graph with $m$ forks shown below, we define the following operator (which appears in $ \hat \sL$):
\begin{equation}\label{opterm}
\mathcal{D}_{\sP_k} := t_{k_{i_{1}}} t_{k_{i_{2}}} \ldots t_{{k_{i_{2m-1}}}} t_{{k_{i_{2m}}}} \frac{\partial}{\partial t_{{k_{i_{1}}}+{k_{i_{2}}-1}}} \ldots   \frac{\partial}{\partial t_{{k_{i_{2m-1}}}+{k_{i_{2m}}-1}}} 
\end{equation}
\begin{figure}[!htb]
\begin{center}
\includegraphics[width=2.8in]{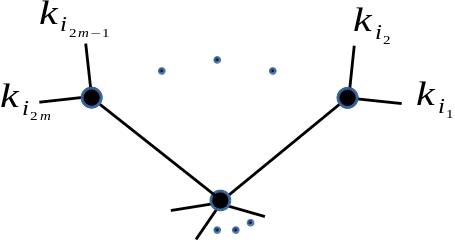}
\end{center}
\caption{${\mathcal{P}}_{k}$-graph with $m$ forks}
\label{PDEOpFig}
\end{figure}
%
By construction, the operators (\ref{opterm}) are in bijection with the $\mathcal{P}_{k}$-graphs. Furthermore, the operator (\ref{opterm}) arises in $ \hat{\sL} $ as a summand in $:(-1)^m\frac{ \sL^m}{m!}: $ with some multiplicity. As part of the proof of theorem, we will show that this multiplicity is $ (-1)^m \frac{1}{|Aut(\hat{\mathcal{P}}_{k})|} $ with $ |Aut(\hat{\mathcal{P}}_{k})| $ as defined in Lemma (\ref{p-pk-Map}). 
%
%
Strategy of proof of theorem (\ref{generatingFnthm}): we will show that for a general $t$-monomial,
its coefficients in $G(\textbf{t})$ and $ \hat \sL (F(\textbf{t})) - \frac{t_0^3}{3!} + 2\frac{t_0^3 t_1}{3!}  $ are equal. For this, we fix a t-monomial, and show a bijection and equality between the summands in the right hand side of (\ref{pkClosedformula}) and the summands in coefficient of the $t$-monomial in $ \hat \sL (F(\textbf{t})) - \frac{t_0^3}{3!} + 2\frac{t_0^3 t_1}{3!}$. This is done via the $\sP_{k}$ graphs as follows. Each $\sP_{k}$ graph tautologically gives a summand in (\ref{pkClosedformula}). Each $\sP_{k}$ graph also gives a term in $\hat{\sL}$ (with some multiplicity) and a monomial in $ F(\textbf{t}) $ whose product is the same $t$-monomial, and its coefficient
in this product agrees with the summand in (\ref{pkClosedformula}). Finally, all the terms in
$ \hat \sL (F(\textbf{t})) - \frac{t_0^3}{3!} + 2\frac{t_0^3 t_1}{3!}  $  are of the form obtained above.
Before the proof, here is an example that illustrates the idea.
\begin{eg}
Consider\\ 
$$ \int_{ \overline{M}_{0,{61}}} \hat \psi_1^{1} \ldots \hat \psi_8^{1} \hat \psi_9^{2} \ldots \hat \psi_{13}^{2}  \hat \psi_{14}^{3} \hat \psi_{15}^{3} \hat \psi_{16}^{4} \hat \psi_{17}^{4} \hat \psi_{18}^{4}  \hat \psi_{19}^{5} \hat \psi_{20}^{5}  \hat \psi_{21}^{6} \hat \psi_{22}^{6}.
$$ 
The corresponding  $t$-monomial in $G(\textbf{t})$ is 
$$\frac{{t_0}^{s_0}{t_1}^{s_1} \ldots {t_l}^{s_l}}{{s_0}!{s_1}! \ldots {s_l}!} = \frac{{t_0}^{39}{t_1}^{8} {t_2}^{5} {t_3}^{2} {t_4}^{3} {t_5}^{2} {t_6}^{2}}  {{39}!{8}! 5! 2! 3! 2! 2!} =: {T}.$$
Now consider the following ${\mathcal{P}}_{k}$-graph in figure (\ref{PDEegFig}) with $m=7$ forks.
\begin{figure}[!htb]
\begin{center}
\includegraphics[width=2in]{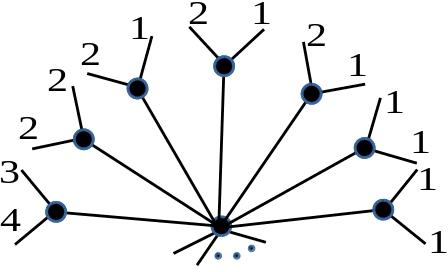}
\caption{${\mathcal{P}}_{k}$-graph with $m=7$ forks}
  \label{PDEegFig}
\end{center}
\end{figure}

The corresponding operator (\ref{opterm}) is
$$
\mathcal{D}_{\sP_k} = t_1^7 t_2^5 t_3 t_4 \frac{\partial^3}{\partial t_{2}^3} \frac{\partial^2}{\partial t_{1}^2} \frac{\partial}{\partial t_{3}} \frac{\partial}{\partial t_{6}}.
$$ 
The coefficient of this operator in $: (-1)^7 \frac{ \sL^7}{7!} :$ is 
$$
 (-1)^7 \left( \frac{ 1}{2^7 7!} \right)  \left( {2^4} {{7}\choose{3,2,1,1}}\right)\\
 = (-1)^7 \left(  \frac{1}{2^3 3! 2!} \right) 
$$
The corresponding unique $t$-monomial $\tilde{T}_{\sP_{k}}$ in $F(\textbf{t})  $ such that $\mathcal{D}_{\sP_k} (\tilde{T}_{\sP_{k}}) = T $ is 
$$\frac{{t_0}^{39}{t_1}^{3} {t_2}^{3} {t_3}^{2} {t_4}^{1} {t_5}^{2} {t_6}^{3}}  {{39}!3! 3! 2! 1! 2! 3!} =: \tilde{T}_{\sP_{k}} $$
Observe that $\tilde{T}_{\sP_{k}}$ can also be read off from the ${\mathcal{P}}_{k}$-graph. 
Now, in  $ \hat \sL (F(\textbf{t})) $, the corresponding term is 
$$
(-1)^7 \left(  \frac{1}{2^3 3! 2!} \right) \left( {\langle {\mathbf{\tau}}^{\mathbf{s}} \rangle}_{\sP_{k}} \right)  \mathcal{D}_{\sP_k} (\tilde{T}_{\sP_{k}})
$$
where  
$$
{\langle {\mathbf{\tau}}^{\mathbf{s}} \rangle}_{\sP_{k}} = \int_{ \overline{M}_{0,{54}}} \hat \psi_1^{1} \hat \psi_2^{1} \hat \psi_3^{1} \hat \psi_4^{2} \hat \psi_6^{2} 
\hat \psi_7^{3} \hat \psi_8^{3} \hat \psi_9^{4} \hat \psi_{10}^{5} \hat \psi_{12}^{5} \hat \psi_{13}^{6} \hat \psi_{14}^{6} \hat \psi_{15}^{6}.
$$
Again, $ {\langle {\mathbf{\tau}}^{\mathbf{s}} \rangle}_{\sP_{k}} $ can also be read off from the ${\mathcal{P}}_{k}$-graph. Also, observe that the coefficient $ (-1)^7 \left(  \frac{1}{2^3 3! 2!} \right) $ of $  \mathcal{D}_{\sP_k} $ in $:(-1)^7\frac{ \sL^7}{7!} : $ is exactly  $ (-1)^m \frac{1}{|Aut(\hat{\mathcal{P}}_{k})|} $ as claimed earlier. Now, 
$$
(-1)^7 \left(  \frac{1}{2^3 3! 2!} \right) \left( {\langle {\mathbf{\tau}}^{\mathbf{s}} \rangle}_{\sP_{k}} \right)  \mathcal{D}_{\sP_k} (\tilde{T}_{\sP_{k}})
$$
$$
= (-1)^7 \left( \frac{1}{2^3 3! 2!} \right) \left( \frac{8!}{(8-7)!} \frac{5!}{(5-5)!} \frac{2!}{(2-1)!} \frac{2!}{(2-1)!} \frac{3!}{(3-1)!}\right)  \left( {\langle {\mathbf{\tau}}^{\mathbf{s}} \rangle}_{\sP_{k}} \right) {T} \\
$$
$$
= (-1)^7 \left( \frac{1}{2^3 3! 2!} \right)  \left( \frac{8! 5! 2! 2!3! }{2!} \right) \left( {\langle {\mathbf{\tau}}^{\mathbf{s}} \rangle}_{\sP_{k}} \right) {T} 
$$
$$
=(-1)^7 \left( \frac{1}{|Aut({\hat{\mathcal{P}}}_{k})|} \right)  \left( \frac{8! 5! 2! 2!3! }{2!} \right) \left( {\langle {\mathbf{\tau}}^{\mathbf{s}} \rangle}_{\sP_{k}} \right) {T}
$$
$$
= - C_{\sP_k} \left( {\langle {\mathbf{\tau}}^{\mathbf{s}} \rangle}_{\sP_{k}} \right) {T}
$$
which gives a summand in the coefficient of $T$ in $ \hat \sL (F(\textbf{t}))- \frac{t_0^3}{3!} + 2\frac{t_0^3 t_1}{3!}$, and this summand agrees with  the summand in (\ref{pkClosedformula}) corresponding to the chosen $\sP_{k}$-graph.
\end{eg}
%
\begin{proof} (of theorem (\ref{generatingFnthm}))\\
Consider a $t$-monomial 
$\frac{{t_0}^{s_0}{t_1}^{s_1} \ldots {t_l}^{s_l}}{{s_0}!{s_1}! \ldots {s_l}!} =: T$. In
$G(\textbf{t})$, $T$ has the coefficient $ \int_{ \overline{M}_{0,{n}}} \hat \psi_1^{k_1}\hat \psi_2^{k_2} \ldots \hat \psi_n^{k_n} $. Now, we will show bijection and equality between the summands in the right hand side of (\ref{pkClosedformula}) and the summands in coefficient of the $T$ in $ \hat \sL (F(\textbf{t}))- \frac{t_0^3}{3!} + 2\frac{t_0^3 t_1}{3!}$.
\begin{enumerate}
\item Pick a $\mathcal{P}_{k}$-graph with no fork. The corresponding operator (\ref{opterm}) in $ \hat \sL $ is $1$, the first term in $ \hat \sL  $, which when applied to $ F(\textbf{t}) $ results in coefficient $1$ for $\frac{{t_0}^{s_0}{t_1}^{s_1} \ldots {t_l}^{s_l}}{{s_0}!{s_1}! \ldots {s_l}!}$ in  $ \hat \sL (F(\textbf{t}))$, which agrees with the corresponding summand in (\ref{pkClosedformula}).
\item Pick a $\mathcal{P}_{k}$-graph with one fork. Without loss of generality, assume a $\mathcal{P}_{k}$-graph with $ k_1, k_2 $ on the only fork, with $ k_1, k_2 $ not simultaneously $0$.
\end{enumerate}
 
Case 1 : $ k_1 \neq k_2$.
Then, the corresponding operator (\ref{opterm}) is $ t_{k_1} t_{k_2} \frac{\partial}{\partial t_{k_1+k_2-1}} $. In $\hat \sL $, this term has coefficient $-1 $. The term in $ F(\textbf{t}) $ that it operates on to produce  $\frac{{t_0}^{s_0}{t_1}^{s_1} \ldots {t_l}^{s_l}}{{s_0}!{s_1}! \ldots {s_l}!}$ has t-monomial: \\
$$\frac{{t_0}^{s_0} \ldots {t_{k_1}}^{s_{k_1-1}} \ldots {t_{k_2}}^{s_{k_2-1}} \ldots {t_{k_l}}^{s_{k_l}}}{{s_0}! \ldots {s_{k_1-1}}! \ldots {s_{k_2-1}}! \ldots {s_l}!} t_{k_1+k_2-1}.$$ 

The result of applying $-t_{k_1} t_{k_2} \frac{\partial}{\partial t_{k_1 + k_2-1}}$ in $\hat{\sL}$ to $F(\textbf{t})$ is the following:  
$$
-t_{k_1} t_{k_2} \frac{\partial}{\partial t_{k_1 + k_2-1}} (\frac{{t_0}^{s_0} \ldots {t_{k_1}}^{s_{k_1-1}} \ldots {t_{k_2}}^{s_{k_2-1}} \ldots {t_{k_l}}^{s_{k_l}}}{{s_0}! \ldots {s_{k_1-1}}! \ldots {s_{k_2-1}}! \ldots {s_l}!} t_{k_1+k_2-1})\\
$$
$$
= - s_{k_1} s_{k_2}\frac{{t_0}^{s_0}{t_1}^{s_1} \ldots {t_l}^{s_l}}{{s_0}!{s_1}! \ldots {s_l}!}. 
$$  
So, the summand contributed by this operator to the coefficient of $\frac{{t_0}^{s_0}{t_1}^{s_1} \ldots {t_l}^{s_l}}{{s_0}!{s_1}! \ldots {s_l}!}$ in $\hat{\sL}(F(\textbf{t}))$ is  
$$
- s_{k_1} s_{k_2} \int_{ \overline{M}_{0,(n-1)}} \psi_{\bullet_1}^{k_1+k_2-1}  \psi_{{3}}^{k_{{3}}}\ldots ... \psi_{r}^{k_{r}} 
$$
and $C_{\sP_k} = s_{k_1} s_{k_2} $ is the number of $ \mathcal{P} $-graphs that map to this $\mathcal{P}_{k}$-graph. This summand also agrees with the corresponding summand in (\ref{pkClosedformula}).
Case 2 : $k_1 = k_2$. In this case we get the corresponding operator (\ref{opterm}) as 
$ t_{k_1}^2  \frac{\partial}{\partial t_{2k_1-1}} $. The coefficient of this term in $\hat \sL$ is $-\frac{1}{2}$.
When $-\frac{1}{2} t_{k_1}^2 \frac{\partial}{\partial t_{2k_1-1}} $ is applied to $ F(\textbf{t}) $, the only terms that produces $\frac{{t_0}^{s_0}{t_1}^{s_1} \ldots {t_l}^{s_l}}{{s_0}!{s_1}! \ldots {s_l}!}$ is  
$$\frac{{t_0}^{s_0}{t_1}^{s_1} \ldots {t_{k_1}}^{s_{k_1-2}} \ldots {t_{k_l}}^{s_{k_l}} }  {{s_0}! \ldots {s_{k_1-2}}! \ldots {s_l}!} t_{2k_1-1} 
$$
Now,
$$
- \frac{1}{2} t_{k_1}^2 \frac{\partial}{\partial t_{2k_1-1}} (\frac{{t_0}^{s_0}{t_1}^{s_1} \ldots {t_l}^{s_l}}{{s_0}!{s_1}! \ldots {s_l}!}\frac{{t_0}^{s_0}{t_1}^{s_1} \ldots {t_{k_1}}^{s_{k_1-2}} \ldots {t_{k_l}}^{s_{k_l}} }  {{s_0}! \ldots {s_{k_1-2}}! \ldots {s_l}!} t_{2k_1-1})\\
$$$$
= - \frac{1}{2} s_{k_1} (s_{k_1}-1) \frac{{t_0}^{s_0}{t_1}^{s_1} \ldots {t_l}^{s_l}}{{s_0}!{s_1}! \ldots {s_l}!} 
$$
So, the the summand contributed by this operator to the coefficient of $\frac{{t_0}^{s_0}{t_1}^{s_1} \ldots {t_l}^{s_l}}{{s_0}!{s_1}! \ldots {s_l}!}$ in $\hat{\sL}(F(\textbf{t}))$ is   
$$
- \frac{1}{2} s_{k_1} (s_{k_1}-1) \int_{ \overline{M}_{0,(n-1)}} \psi_{\bullet_1}^{2k_1-1}  \psi_{{3}}^{k_{{3}}}\ldots ... \psi_{r}^{k_{r}} 
$$
and $C_{\sP_k} = \frac{1}{2} s_{k_1} (s_{k_1}-1) $ is the number of $ \mathcal{P} $-graphs that map to this kind of $\mathcal{P}_{k}$-graph. This summand also agrees with the corresponding summand in (\ref{pkClosedformula}). 

In both cases, the terms contribute $- C_{\sP_k} \int_{ \overline{M}_{0,(n-1)}} \psi_{\bullet_1}^{k_{1}+k_{2}-1} \psi_{{3}}^{k_{{3}}}\ldots ... \psi_{r}^{k_{r}} 
$ to the coefficient of t-term $\frac{{t_0}^{s_0}{t_1}^{s_1} \ldots {t_l}^{s_l}}{{s_0}!{s_1}! \ldots {s_l}!}$ in $ \hat{\sL} (F(\textbf{t})) $. So, we get both summands in (\ref{pkClosedformula}) corresponding to two $\sP_k$-graphs with one fork. Also, observe that if $k_1 = k_2 = 0$, the term $- t_{0}^2 \frac{\partial}{\partial t_{-1}} $ contributes nothing, and there is also no corresponding summand in (\ref{pkClosedformula}).

Now consider a $\mathcal{P}_{k}$-graph with $m$ forks. Without loss of generality, let the $k_i$'s on the forks be $ \lbrace k_1, k_2, \ldots, k_{2m} \rbrace $ as shown in figure below. Let $l_{i}$ be the number of times a given $k_i$ appears on any fork on the $\mathcal{P}_{k}$-graph, and let $ j_1, j_2, \ldots, j_f $ be the number of forks with the same set of $k_i$'s on them; and let $d$ be the number of forks with both $k_i$'s the same on that fork.
\begin{figure}
\begin{center}
\includegraphics[width=2.7in]{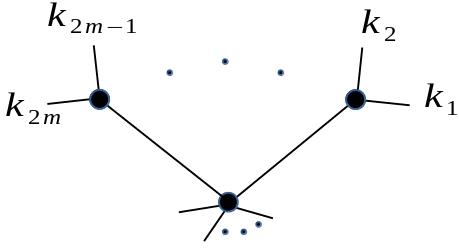}\\
\end{center}
\caption{$\mathcal{P}_{k}$-graph with $m$ number of forks}
\label{figProof4.1a}
\end{figure} 
Then the corresponding operator (\ref{opterm}) is
$$ 
t_{k_1} t_{k_2} \ldots  t_{k_{2m-1}} t_{k_{2m}} \frac{\partial}{\partial t_{k_1+k_2-1}}  \ldots \frac{\partial}{\partial t_{k_{2m-1}+t_{k_{2m}}-1}} := \mathcal{D}_{\mathcal{P}_{k}}.
$$
The coefficient of this term in $\hat \sL$ as a summand in  
 $:(-1)^m\frac{ \sL^m}{m!}: $ is given by
$$
 (-1)^m \left( \frac{ 1}{2^m m!} \right)  \left( {2^{(m-d)}} {{m}\choose{l_1,l_2,\ldots,l_n}}\right)\\
 = (-1)^m \left(  \frac{1}{2^d l_1! l_2! \ldots,l_n!} \right). 
$$
The corresponding term in $F(\textbf{t})  $ is 
$$\frac{{t_0}^{s_0}\ldots {t_{k_1}}^{\left( {s_{k_1}- l_{k_1}}\right)}  \ldots {t_{k_{2m}}}^{\left( {s_{k_{2m}}- l_{k_{2m}}}\right)} }  {{s_0}! \ldots {\left( {s_{k_1}- l_{k_1}}\right)}!  \ldots {\left( {s_{k_{2m}}- l_{k_{2m}}}\right)} ! \ldots {s_{k_{n}} }!}
t_{k_1+k_2-1} \ldots t_{k_{2m-1}+k_{2m}-1}
=: \tilde{T}_{\sP_{k}}. $$
In  $ \hat \sL (F(\textbf{t})) $, the corresponding term is 
$$
(-1)^m \left(  \frac{1}{2^d l_1! l_2! \ldots,l_n!} \right)  \left( \langle {\mathbf{\tau}}^{\mathbf{s}} \rangle_{\sP_{k}} \right)  \mathcal{D}_{\sP_k} (\tilde{T}_{\sP_{k}})
$$
where  
$\langle {\mathbf{\tau}}^{\mathbf{s}} \rangle_{\sP_{k}}$ is the $\psi$-monomial on $\overline{M}_{0,{(n-m)}}$ that appears as coefficient of $ \tilde{T}_{\sP_{k}} $ in $ F(\textbf{t}) $. 
Observe that the coefficient $ (-1)^m \left(  \frac{1}{2^d l_1! l_2! \ldots,l_n!} \right)  $ of $  \mathcal{D}_{\sP_k} $ in $:(-1)^m\frac{ \sL^m}{m!}: $ is exactly  $ (-1)^m \frac{1}{|Aut(\hat{\mathcal{P}}_{k})|} $ as claimed earlier. Now, 
$$
(-1)^m \left(  \frac{1}{2^d l_1! l_2! \ldots,l_n!} \right)  \left( \langle {\mathbf{\tau}}^{\mathbf{s}} \rangle_{\sP_{k}} \right)  \mathcal{D}_{\sP_k} (\tilde{T}_{\sP_{k}})
$$
$$
= (-1)^m \left(  \frac{1}{2^d l_1! l_2! \ldots,l_n!} \right)  \left( \frac{s_{1}!}{\left( s_{1}-l_{1}\right)! } \frac{s_{2}!}{\left( s_{2}-l_{2}\right)! } \ldots \frac{s_{n}!}{\left( s_{n}-l_{n}\right)! } \right)  \left( \langle {\mathbf{\tau}}^{\mathbf{s}} \rangle_{\sP_{k}} \right) {T} \\
$$
$$
= (-1)^m \left( \frac{1}{|Aut({\hat{\mathcal{P}}}_{k})|} \right)  \left( \frac{s_{1}!}{\left( s_{1}-l_{1}\right)! } \frac{s_{2}!}{\left( s_{2}-l_{2}\right)! } \ldots \frac{s_{n}!}{\left( s_{n}-l_{n}\right)! } \right)  \left( \langle {\mathbf{\tau}}^{\mathbf{s}} \rangle_{\sP_{k}} \right) {T} \\
$$
$$
= (-1)^m C_{\sP_k} \left( \langle {\mathbf{\tau}}^{\mathbf{s}} \rangle_{\sP_{k}} \right) {T}
$$
which gives a summand in the coefficient of $T$ in $ \hat \sL (F(\textbf{t})) $ and this agrees with  the summand in (\ref{pkClosedformula}) corresponding to the chosen $\sP_{k}$-graph.

So, one direction is proved. To show bijection in the other direction, we pick a summand in the coefficient of $ \frac{{t_0}^{s_0}{t_1}^{s_1} \ldots {t_l}^{s_l}}{{s_0}!{s_1}! \ldots {s_l}!} = T$ in $ \hat \sL (F(\textbf{t})) $ that comes from term $ :\frac{(-1)^{m}}{m!} {\sL}^m :$. Let this summand come from the following summand in $ :\frac{(-1)^{m}}{m!} {\sL}^m$ :
$$
(-1)^m \left(  \frac{1}{2^d l_1! l_2! \ldots,l_n!} \right)  \mathcal{D} 
$$
where 
$$ 
\mathcal{D} = t_{k_1} t_{k_2} \ldots  t_{k_{2m-1}} t_{k_{2m}} \frac{\partial}{\partial t_{k_1+k_2-1}}  \ldots \frac{\partial}{\partial t_{k_{2m-1}+t_{k_{2m}}-1}}.
$$
Now, $ \mathcal{D} = \mathcal{D}_{\sP_{k}} $, with ${\sP_{k}}$ as in the figure (\ref{figProof4.1a}).
Now, $ \langle {\mathbf{\tau}}^{\mathbf{s}} \rangle_{\sP_{k}} $ is uniquely determined by $ \tilde{T}_{\sP_{k}} $, and $ \tilde{T}_{\sP_{k}} $ is uniquely determined by the relation
$$
\mathcal{D}_{\sP_k} (\tilde{T}_{\sP_{k}}) = T.
$$
Then the chosen summand in the coefficient of $T$ in $ \hat \sL (F(\textbf{t})) $ is 
$$
= (-1)^m C_{\sP_k} \left( \langle {\mathbf{\tau}}^{\mathbf{s}} \rangle_{\sP_{k}} \right). 
$$
The $\sP_k$-graph thus determined also determines the same summand in (\ref{pkClosedformula}).

Finally, $ \hat \sL (F(\textbf{t})) $ also contains two monomials with degree $\leq 4$, and these are  
$ \frac{t_0^3}{3!} $ and $- 2\frac{t_0^3 t_1}{3!}$. We subtract these monomials as the Hassett Spaces for the corresponding cases of three and four pointed rational curves with all weights $\frac{1}{2}$ are empty, as these rational curves are not $\mathcal{A}$-stable.
\end{proof}

\section*{Acknowledgements} 
The author would like to thank the anonymous reviewer for constructive criticism and helpful suggestions.

\bibliographystyle{plain}
\bibliography{proposal.bib}
\end{document}